\documentclass[11pt,a4paper]{article}

\usepackage{hyperref}

\usepackage{soul}

\usepackage{dsfont}
\usepackage{amsfonts}
\usepackage{amsmath}
\usepackage{amssymb}
\usepackage[english]{babel}
\usepackage[T1]{fontenc}
\usepackage{graphicx}
\usepackage{latexsym}
\usepackage{amstext}
\usepackage{amsthm}
\usepackage{mathrsfs}
\usepackage{graphics}
\usepackage{enumerate}
\usepackage{paralist}
\usepackage{color}
\usepackage{fullpage}

\DeclareFontFamily{OT1}{pzc}{}
\DeclareFontShape{OT1}{pzc}{m}{it}{<-> s * [1.10] pzcmi7t}{}
\DeclareMathAlphabet{\mathpzc}{OT1}{pzc}{m}{it}

\newtheorem{prop}{Proposition}[section]
\newtheorem{rem}[prop]{Remark}
\newtheorem{lem}[prop]{Lemma}
\newtheorem{theo}[prop]{Theorem}

\numberwithin{equation}{section}

\newcommand{\beq}{\begin{eqnarray}}
\newcommand{\beqq}{\begin{eqnarray*}}
\newcommand{\eeq}{\end{eqnarray}}
\newcommand{\eeqq}{\end{eqnarray*}}

\newcommand{\R}{\mathbb{R}}
\newcommand{\E}{\mathbb{E}}
\renewcommand{\P}{\mathbb{P}}

\newcommand{\Sub}{\mathpzc{k \hspace{0.02cm}}}
\newcommand{\Sb}{\mathpzc{b}}
\newcommand{\X}{\mathds{X}}
\newcommand{\Xb}{\X^{(\Sb)}}
\newcommand{\Xbs}{X^{(\Sb_{\Sub})}}
\newcommand{\bfXs}{\mathbf{X}^{\Sub}}

\newcommand{\C}{\mathbb{C}}
\newcommand{\Lt}{\ell}
\newcommand{\T}{\mathds{T}}
\newcommand{\Tm}{{T}}
\newcommand{\Rs}{\mathfrak{R}}
\newcommand{\ST}[2]{#1_a(#2)}

\newcommand{\Lev}{{X}}
\newcommand{\dk}{{\mathrm{d}}_{{\Sub}}\hspace{0.02cm}}
\newcommand{\dt}{\mathrm{d}_{\Sb}\hspace{0.02cm}}

\newcommand{\g}{\mathtt{g}}
\renewcommand{\d}{\mathtt{d}}
\renewcommand{\Xi}{\mathbb{C}_{\overline{\mathbf{\Psi}}}}

\newcommand{\Ikea}{{\Sub_{\Sb}^{\triangleright q}}}

\title{First passage times over stochastic boundaries   for subdiffusive processes} %

\author{{\sc  C. Constantinescu}\thanks{University of Liverpool, Department of Mathematical Sciences, Liverpool, L69 3 BX, UK. \ \ E-mail: 	C.Constantinescu@liverpool.ac.uk }, \ {\sc R. Loeffen}\thanks{University of Manchester, School of Mathematics,  Manchester, M13 9PL, UK. \ \ E-mail: ronnie.loeffen@manchester.ac.uk } \  {\sc and }
  {\sc P. Patie}\thanks{Cornell University, School of Operations Research and Information Engineering,
   220 Rhodes Hall, Ithaca, NY 14853, U.S.A. \ \ E-mail: ppatie@cornell.edu }}
   \date{}

\begin{document}
\maketitle
\begin{abstract}
Let $\X=(\X_t)_{t\geq 0}$ be the subdiffusive  process defined,  for any $t\geq 0$, by $ \X_t = \Lev_{\Lt_t}$ where  $\Lev=(\Lev_t)_{t\geq 0}$ is a L\'evy process and $\Lt_t=\inf \{s>0;\: \Sub_s>t \}$ with  $\Sub=(\Sub_t)_{t\geq 0}$  a subordinator independent of $\Lev$.  We start by developing  a composite Wiener-Hopf factorization to characterize  the  law  of  the pair $(\T_a^{(\Sb)}, (\X - \Sb)_{\T_a^{(\Sb)}})$ where
 \begin{equation*}%
\T_a^{(\Sb)} = \inf \{t>0;\: \X_t > a+ \Sb_t  \}
\end{equation*}
with $a \in \R$ and $\Sb=(\Sb_t)_{t\geq 0}$  a (possibly degenerate) subordinator independent of $\Lev$ and $\Sub$.
We proceed by providing a detailed analysis of the cases where either $\Sub$ is a stable subordinator or $X$ is spectrally negative.
Our proofs hinge on a variety of techniques including excursion theory, change of measure, asymptotic analysis and on establishing a link between subdiffusive processes and a subclass of semi-regenerative processes. In particular, we show that the variable $\T_a^{(\Sb)}$ has the same law as the first passage time of a  \emph{semi-regenerative process of L\'evy type}, a terminology that we introduce to  mean that this process satisfies the  Markov property of L\'evy processes for stopping times whose graph is included in the associated regeneration set.
\end{abstract}

\textbf{AMS 2010 subject classifications:} Primary: 60K15, 60G40. Secondary: 60G51, 60G52,  60G18.
$\vspace{5pt}$

\section{Introduction and main results}\label{sec_intro}

The recent years have witnessed strong and steady  interests in the analysis of subdiffusive  dynamics constructed by time-changing a Brownian motion, a Poisson process or  any  L\'evy process, by the inverse of an independent subordinator. On the one hand, this seems to be attributed  to the appearance of such dynamics in some important limit theorems such as the scaling limit of continuous-time random walks (in which the i.i.d.~jumps are
separated by i.i.d.~waiting times) \cite{M2},  the scaling limit of random walks in random environment \cite{Ben},  and also the (surprising) intermediate time behaviour of some periodic diffusive flows  which gives rise to the fractional kinetic process \cite{Hairer}. Moreover, in functional analysis, these processes appear in the stochastic representation of solutions of fractional Cauchy problems defined similarly to the classical Cauchy problem by replacing the time derivative by the fractional one or more generally by some convolution type operators \cite{Baeu}.   On the other hand, the statistical properties of these models, e.g.~subdiffusive behavior, long-range dependence, fractal properties, see \cite{Schil-Cor},  enable to reproduce some complex phenomena that have been observed in physical sciences  such as in statistical physics \cite{Pyr}, chemical physics \cite{ChemPhys}, see also the paper \cite{Review} for an excellent description of the ubiquity of anomalous transport in nature  and, also in economy, see \cite{Leo-Fin}. %

One generic and important  theme of research with various applications in this area, and more generally for non-Markovian dynamics, is  the first passage time problem. There is indeed a rich and substantial literature devoted to the study and applications of this problem in the context of  Gaussian processes,  semi-Markov  processes,   and   some self-similar non-Markovian  processes, see for instance \cite{Dembo, Nualart}, \cite{Toaldo-FPT, Deng, Guo}, \cite{LPS} and the references therein.  However, unlike for Markov processes, this literature reveals that the non-Markovian property  make the analysis of such objects very difficult  and,  in general,  only  very partial statistical information regarding these random variables has been obtained. An interesting feature worth mentioning  is the  phenomenon of persistency  that has been observed   for some Gaussian processes and for self-similar non-Markovian processes, meaning  that the survival probabilities of the first passage time distribution has a power decay which is independent of the state variable, see
 \cite{Dembo, Toaldo-FPT, LPS} and the references therein. We already point out that, as a by-product of our results, we shall also identify, in the case where the subordinator driving the time change is a stable subordinator,  a refinement of the persistence phenomena for the distribution of the first passage time for subdiffusive processes, see Proposition \ref{cor:ident}.

The main objective of this paper is to offer a fresh perspective on this issue by  establishing a general theory for  the first passage time problem over a stochastic boundary, not only a fixed one, of the  class of subdiffusive processes that we now introduce. Throughout,
we consider the  stochastic process $\X=(\X_t)_{t\geq 0}$  defined on the probability space $(\Omega,\mathcal F,\P)$ %
by
\beq\label{def:reglevy}
\X_t = \Lev_{\Lt_t}, \quad t\geq 0,
\eeq
where
$\Lt_t=\inf \{s>0;\: \Sub_s>t \}$ with  $\Sub=(\Sub_t)_{t\geq 0}$  a subordinator, issued from $0$, associated with the Bernstein function $\phi_{\Sub}$ and  $\Lev=(\Lev_t)_{t\geq 0}$ is an independent  L\'evy process with characteristic exponent $\Psi$, all these notion will be reviewed below.
Our first main contribution  is  to provide, by means of a composite Wiener-Hopf factorization, an explicit characterization of  the joint law of $(\T_a^{(\Sb)}, (\X - \Sb)_{\T_a^{(\Sb)}})$, where  $\T_a^{(\Sb)}$ is the first passage time of $\X$ to the stochastic boundary $a+ \Sb$, defined as
\beq\label{defT}
\T_a^{(\Sb)} = \inf \{t>0;\: \X_t > a+ \Sb_t  \}
\eeq
with $a \in \R$ and $\Sb=(\Sb_t)_{t\geq 0}$ is another subordinator, associated to the  Bernstein function  $\phi_{\Sb}$,  defined on $(\Omega,\mathcal F,\mathbb P)$ and assumed to be independent of   $\Lev$ and $\Sub$, and thus of $\X$.
We shall also give more insights into the fractional subdiffusive class, that is when $\Sub$ is a stable subordinator, and also into the case where $\X$ does not have positive jumps.

To the best of our knowledge, up to now only in very few  isolated cases, explicit expressions have been found for the Laplace transform or simply about the mass at infinity of $\T_a^{(\Sb)}$ and they fall under two categories. One where $\Sub$ is a stable subordinator, $\Lev$ is a Brownian motion plus drift or a stable process   and $\Sb=0$, see \cite{Deng, Guo, LPS}. The second category corresponds to $\Lev$ being a compound Poisson as then $\X$ becomes a compound renewal process and the first passage times over in particular affine barriers for the latter class of processes have been well studied in the literature, see \cite{AA} for an overview. Our results, which  allow for explicit expressions of several quantities of interest for the  first passage time problem of general subdiffusive processes,   thus provide a vast improvement of the existing literature. Our paper  also  strengthens  their tractability and hence their applicability as models for complex phenomena.

  We also point out that the stochastic boundary $a+ \Sb_t$ boils down to an affine curve, with a positive slope, when $\Sb_t = \dt t, \dt>0,$ is a degenerate pure drift subordinator, that is
\begin{equation*}
\T_a^{(\dt)} = \inf \{t>0;\: \X_t > a+ \dt t  \}
\end{equation*}
where $a\in\R$. We highlight this case since it corresponds to a generalization of the  models that have been used recently in risk theory,  where the parameter $\dt$ corresponds  to the constant risk premium, see \cite{Biard-Saussereau-2014,rockpaper} and the references therein. Note that the stochastic boundary case is  also of interest in this context, see the discussion of the example in Section \ref{sec:cox}.
We mention that our analysis covers also the dual first passage time
\begin{equation}\label{def:Ta}
\widehat{\T}_a^{(\Sb)}  = \inf \{t>0;\: \X_t < -a - \Sb_t  \}
\end{equation}
 by observing that $\widehat{\T}_a^{(\Sb)} = \inf \{t>0;\: -\X_t > a+ \Sb_t  \}$ and $-\X$ is simply the time change still by $\Lt$ of the dual L\'evy process $\widehat \Lev = -\Lev$ which is another L\'evy process. Finally anticipating the discussion shortly after, it  follows from  the definition of $\X$~and the spatial homogeneity of the L\'evy process $\Lev$ that for any $x \in \R$, the law of  $(\X_t)_{t\geq 0}$ under $\P_x$ (i.e.~$\P_x(X_0=x)=1$) is the law of $(\X_t+x)_{t\geq 0}$ under $\P=\P_0$. In particular, $\T_a^{(\Sb)}$ under $\P_x$ has the same law as $\T_{a-x}^{(\Sb)}$ under $\P$.

Let us now recall that $\Sub$ and $\Lev$ as L\'evy processes are stochastic processes with stationary and independent increments, with a.s.~c\`adl\`ag sample paths and their law are fully characterized by their Laplace exponent $\phi_{\Sub}(u)= -\log \E[e^{-u \Sub_1}], u\geq 0$, and characteristic function, $\Psi(z)  =\log \E[e^{iz \Lev_1}],z \in \R$, that take respectively the form
\beq\label{LKs}
\phi_{\Sub}(u)=  {\dk}u   + \int_{0}^{\infty} (1 - e^{-uy}) \mu(dy), \quad u\geq0,
\eeq
where $\dk \geq 0$ and $\mu$ is a L\'evy measure such that $\int_{0}^{\infty} (1\wedge y) \mu(dy)<+\infty$ and
\beq\label{LK}
\Psi(z)= -\frac{\sigma^2}{2} z^2 +i \mathrm{d}_X \hspace{0.01cm} z + \int_{\R}(e^{i z y} -1-iy z \mathbb{I}_{\{|y|<1\}})\Pi(dy),
\eeq
in which $\sigma^2\geq 0$, $\mathrm{d}_X \in\mathbb{R}$, the coefficient of the drift part and $\Pi$ is the L\'{e}vy measure that characterizes the jumps and satisfies the condition $\int_{\R}(1\wedge |y|^{2})\Pi(dy)<+\infty$ and $\Pi(\{0\})=0$. To avoid having to treat a less interesting case separately, we assume throughout the paper  that $\Sub$  is not a compound Poisson process that is
\begin{equation}\label{eq:ass}
 \textrm{ ${\dk}>0$  or $\mu(\mathbb R^+)=\infty$ in \eqref{LKs}.}
 \end{equation}
This entails that  $\Sub$ has a.s.~increasing, not just non-decreasing, sample paths and thus  the trajectories of $\Lt$ are a.s.~non-decreasing, continuous and when $\dk=0$ they are also singular with respect to the Lebesgue measure as  the closure of the range of $\Sub$ has zero Lebesgue
measure in this case. Note also that intervals on which $\Lt$ is constant correspond to intervals that $\Sub$ jumps over, assuming that $\Sub$ is not of course a  pure drift. Thus, we have,  for  any $t>0$ and  small $h>0$, that the value of $\ell_{t+h}$  clearly depends on whether $t$ is in an interval on which $\ell$ is constant or not. Since the latter information is known when given the history of the process $\Lt$ up to time $t$ but not when given just the value of $\Lt_t$, it follows that $\Lt$ is not a Markov process. Since $\Lev$ and $\Lt$ are independent it follows that $\X$ is also not Markov, unless $\Lev=0$ or $\Sub$ is a pure drift.
It turns out that the subdiffusive processes we consider in this paper are connected to some substantial classes of non-Markovian processes that have been studied in the literature, namely the semi-regenerative and Cox and renewal processes. In view of the importance of subdiffusive processes, we believe that it is worth mentioning them in the sequel and we emphasize that the connection with semi-regenerative processes will be essential  in the proof of our results.

\bigskip

\noindent \textbf{Subdiffusive processes as semi-regenerative processes of L\'evy type.}
The process $\X$ is not Markov and so certainly is not Markov at any stopping time. However,  given a suitable filtration, we show in the following that the  Markov property still holds for stopping times $T$ that take values in the random set
\begin{equation}\label{regen_set}
\Rs=\{t\geq 0; \:\Sub_{\Lt_t}=t\}.
\end{equation}
Before stating this result we introduce some notation. Throughout, we denote   by  $(\mathcal F_t)_{t\geq 0}$ the natural filtration of the bivariate L\'evy process $(X,\Sub)$ and by  $\mathcal N=\{A\in\mathcal F; \: \mathbb P(A)=0\}$ the collection of null-events, and for all $t\geq 0$, by $\mathcal F^\P_t = \sigma(\mathcal F_t\cup \mathcal N)$ the smallest $\sigma$-algebra containing $\mathcal F_t\cup \mathcal N$ and write $\widetilde{\mathcal{F}}_t =\mathcal F^\P_{\ell_t} $.   We shall also need the notion of regeneration sets and semi-regenerative processes, whose formal definition in the canonical setting can be found in Maisonneuve \cite[Chap.~2]{Maison}. For our purpose, we consider, within the class of  regeneration sets,  the  right-closed random subsets of $[0,\infty)$ that have the same law as  the range of a subordinator. Intuitively, a semi-regenerative process regenerates at every stopping times that belong to its associated regenerative set, albeit possibly from a different starting position, which is why in the literature it is called semi-regenerative rather than regenerative.   %
Finally, we recall that  for a filtration $(\mathcal G_t)_{t\geq 0}$ on some measurable space $(\Omega,\mathcal G)$ and a $(\mathcal G_t)_{t\geq 0}$-stopping $T$ the $\sigma$-algebra $\mathcal G_T$ is defined by
$\mathcal G_T = \{A\in\mathcal G;\: A\cap\{T\leq t\}\in \mathcal G_t \ \text{for all $t\geq 0$}\}  $.

\begin{prop}\label{prop:regenerative}
 The filtration $(\widetilde{\mathcal{F}}_t)_{t\geq 0}$ is  well-defined,  right-continuous and $\X$ is adapted to it. Furthermore, for any $(\widetilde{\mathcal{F}}_t)_{t\geq 0}$-stopping time $T$  such that its graph  $[T]=\{(t,\omega); \: t=T(\omega)<\infty\} \subseteq \Rs$ a.s.~and satisfying $\mathbb P(T<\infty)>0$, we have that, under    $\mathbb P(\cdot|T<\infty)$,   the process
\begin{equation}	 \label{eq:SMPL}
(\X_{T+t}-\X_T)_{t\geq 0}  \textrm{ is independent of } \widetilde{\mathcal{F}}_T \textrm{ and has the same law as } \X \textrm{ under } \mathbb P.
\end{equation}
 In particular, $\X$ is a semi-regenerative process associated to the regeneration set $\mathfrak R$, that is  for any $T$  an $(\widetilde{\mathcal{F}}_t)_{t\geq 0}$-stopping time  such that  $[T] \subseteq \Rs$,   $G$ a positive and $\mathcal{F}^0=\sigma(\X_t, t\geq 0)$-measurable function %
then, for any $t\geq0$, $
\E\left[ G(\X_{t+T}) | \widetilde{\mathcal{F}}_T\right]= \E_{\mathds X_{T}}\left[ G(\X_{t})\right] \: \P\textrm{-a.s.~on } \{T<\infty\}$.
\end{prop}
\begin{rem}
When $\X$ is a L\'evy process, then the property \eqref{eq:SMPL} is called the strong Markov property of L\'evy processes which implies  the classical strong Markov property. The situation is similar here in the context of semi-regenerative processes and we call $\X$ a semi-regenerative process of L\'evy type. %
\end{rem}
The proof of the proposition is postponed to Section \ref{sec:proofprop1} below. Actually we shall prove this statement  in a more general setting where $\X$ in   \eqref{def:reglevy} can be defined from  a bivariate L\'evy process  $(\Lev,\Sub)$ with possibly dependent components.   This extended version of the connection with semi-regenerative processes will be essential to show, in the spectrally negative case, that the process of passage times $(\T_a^{(\Sb)})_{a\geq0}$  is, under $\P$, a subordinator, see Proposition \ref{cor:sn}.

\bigskip

\noindent \textbf{Subdiffusive Poisson processes as Cox and renewal processes.} A Cox process or a  doubly stochastic Poisson process is a Poisson process time-changed by an independent stochastic process with non-decreasing and right-continuous sample paths starting at 0, see e.g.~\cite[p.~11]{Grandell-1976}. On the other hand, a renewal process is a continuous time Markov chain starting at $0$, with jumps of size $+1$ and i.i.d.~holding times which we assume are $(0,\infty)$-valued. By a result of Kingman,  it follows that a Cox process is a renewal process if and only if the underlying time-change process is the inverse of a subordinator with increasing sample paths, see \cite[p.35]{Grandell-1976}. In this case the corresponding holding time distribution $F$ of the Cox-renewal process is characterized by
\begin{equation}\label{assump_hold}
\int_0^\infty e^{-u x} F(dx) = \frac{\lambda}{\lambda+\phi(u)}, \quad u \geq 0,
\end{equation}
where $\phi$ is the Laplace exponent of the subordinator driving the time-change and $\lambda>0$ is the intensity of the Poisson process.
Note that   Yannaros \cite[Lemma 2.1]{Yannaros-1994} showed that a sufficient condition for a distribution $F$ to be of the form \eqref{assump_hold} is that it  is absolutely continuous with a completely monotone density. However,  from  \cite[Remark 11.13]{Schilling-Song-Vondracek-2012} and the discussion above, we deduce readily this refined result.
\begin{prop}
Let us assume that $X$ is a Poisson process. Then, the subdiffusive process $\X$ defined in  \eqref{def:reglevy} is both a Cox and a  renewal process. Moreover, $F$ is of the form \eqref{assump_hold}  if it has a probability density which is log-convex and so then the renewal process with holding time distribution $F$ is of the form \eqref{def:reglevy}.
\end{prop}
A popular example in this literature  is when the subordinator is an $\alpha$-stable subordinator, i.e. $\phi(u)=u^\alpha$ with $\alpha\in(0,1)$. The resulting Cox-renewal process is referred to as the fractional Poisson process and the holding time distribution $F$ has the Mittag-Leffler distribution  given by
\begin{equation}\label{mittag-leffler_distr}
F(x) = 1 - \textrm{E}_\alpha(-\lambda x^\alpha), \quad x\geq 0,
\end{equation}
where $\textrm{E}_\alpha(x)=\sum_{n= 0}^{\infty} \frac{x^n}{\Gamma(\alpha n+1)}$ is the Mittag-Leffler function.%

\bigskip

In the remaining part of this section, we state the main results on the first passage time problems. In Section \ref{sec:ex}, we illustrate our approach  by detailing two examples. The last section is devoted to the proof of the results of Section \ref{sec_intro}.

\subsection{The stochastic boundary  via a composite of Wiener-Hopf factorizations}\label{sec:fpt}

Before stating our first main result we recall some information regarding the Wiener-Hopf factorization of L\'evy processes and refer to \cite[Section 45]{Sato1999} for a nice exposition.
Throughout  ${\bf{e}}_p$ stands for an exponential random variable with parameter $p>0$ which is independent of the triple of processes $(\Lev,\Sub,\Sb)$.
Recall that $X$ is a L\'evy process with characteristic exponent $\Psi$. From the Wiener-Hopf factorization we know that, for any $p>0$, there exists two functions  ${\Phi}(p;z)$ and ${\widehat \Phi}(p;z)$ such that

\beq\label{def_WHfactors}
\frac{p}{p-\Psi(z)}={\Phi}(p;z){\widehat \Phi}(p;z), \quad z\in\mathbb R,
\eeq
where, with $\overline X_t=\sup_{0\leq s\leq t}X_s$ and $\underline X_t=\inf_{0\leq s\leq t}X_s$, the Wiener-Hopf factors ${\Phi}$ and ${\widehat \Phi}$ are identified as
\begin{equation*}
{\Phi}(p;z) = \mathbb E \left[ e^{i z \overline X_{{\bf{e}}_p}} \right], \quad {\widehat \Phi}(p;z) =  \mathbb E \left[ e^{i z \underline X_{{\bf{e}}_p}} \right].
\end{equation*}
 Note that $z\mapsto{\Phi}(p;z)$ can be analytically extended to $\Im(z)\geq0$. Moreover, if $X$ drifts to $-\infty$, i.e. $\lim_{t\to\infty}X_t=-\infty$ a.s., then  $\overline X_\infty:=\lim_{t\to\infty}\overline X_t<\infty$ a.s. and it is then not hard to show that
${\Phi}(0;z):=\lim_{p\downarrow0}{\Phi}(p;z)=\mathbb E \left[ e^{i z \overline X_{\infty}} \right]$ for any $\Im(z)\geq0$. We are now ready to state our first main result.

\begin{theo}\label{thm2}\label{thm1}
For any $q\geq 0$, the mapping
\begin{equation}\label{eq:Psidef}
z \mapsto
  \Psi_{\Ikea}(z) = \Psi(z)-\phi_{\Sub}(\phi_{\Sb}(iz)+q)+\phi_{\Sub}(q)
  \end{equation}
  is the characteristic exponent of a L\'evy process.
  Moreover,
 writing ${\Phi}_{\Ikea}$ and $\widehat{\Phi}_{\Ikea}$ for its Wiener-Hopf factors, we have, for any $q>0$, $p>0$, $v\geq 0$ with $p\neq v$,
\beq\label{main}
\E \left[ e^{ -q\T_{{\bf{e}}_p}^{(\Sb)} - v ( \X^{(\Sb)}_{\T_{{\bf{e}}_p}^{(\Sb)}} - \bf{e}_p ) } \right] = \frac{p}{p-v}\left(1-\frac{{\Phi}_{\Ikea}(\phi_{\Sub}(q);ip)}{{\Phi}_{\Ikea}(\phi_{\Sub}(q);iv)}\right)
\eeq
where we have set $\X^{(\Sb)}=(\X^{(\Sb)}_t=\X_t-\Sb_t)_{t\geq0}$.
If the L\'evy process with characteristic exponent $\Psi_{{\Sub_{\Sb}^{\triangleright 0}}}$ drifts to $-\infty$, then %
for any $p>0$, $v\geq 0$ with $p\neq v$,
\beq\label{main_qis0}
\E \left[ e^{ - v ( \X^{(\Sb)}_{\T_{{\bf{e}}_p}^{(\Sb)}} - \bf{e}_p ) } \mathbb I_{\{\T_{{\bf{e}}_p}^{(\Sb)} <\infty \}} \right] = \frac{p}{p-v}\left(1-\frac{{\Phi}_{{\Sub_{\Sb}^{\triangleright 0}}}(0;ip)}{{\Phi}_{{\Sub_{\Sb}^{\triangleright 0}}}(0;iv)}\right).
\eeq
If the L\'evy process with characteristic exponent $\Psi_{{\Sub_{\Sb}^{\triangleright 0}}}$ does not drift to $-\infty$, then $\T_{a}^{(\Sb)} <\infty$ a.s. for any $a\geq 0$. %
 \end{theo}
 \begin{rem}
   Note that when $\Sb=0$, the composite Wiener-Hopf factorization \eqref{main} reduces to a subordinate Wiener-Hopf factorization as in this case $\Psi_{\Ikea} = \Psi$ and thus we have, for any $q>0$, $p>0$, $v\geq 0$ with $p\neq v$ and writing simply $\T_{a}=\T^{0}_{a}$,
\begin{equation*}%
\E \left[ e^{ -q\T_{{\bf{e}}_p} - v ( \X_{\T_{{\bf{e}}_p}} - \bf{e}_p ) } \right] = \frac{p}{p-v}\left(1-\frac{{\Phi}(\phi_{\Sub}(q);ip)}{{\Phi}(\phi_{\Sub}(q);iv)}\right).
\end{equation*}
As explained at the beginning of Section \ref{sec:proofmain} this result is obtained rather easily from a classical  time change technique whereas the proof of the general case requires a more refined analysis.
 \end{rem}

\bigskip

 We proceed by providing some interesting by-products and refined results of Theorem \ref{thm1} for two substantial   classes of semi-regenerative processes of L\'evy type, namely the spectrally negative and fractional ones.

 \subsection{Spectrally negative subdiffusive processes}

 We start by carrying out an in-depth analysis of the interesting class of semi-regenerative processes that creep upward, that is when the processes hit point above the starting point continuously. We refer  to them  as the class of spectrally negative semi-regenerative processes of L\'evy type. Since $\Lt$ has continuous paths,  this class is in bijection with the one of spectrally negative L\'evy processes and we refer to \cite[Sec.~46]{Sato1999}, \cite[Chap.~VII]{Bertoin-96} and \cite[Chapter 8]{Kyprianou-14} for a thorough account on these processes. In particular, such a  L\'evy process $\Lev$ does not experience any positive jumps i.e.~$\P(\Delta \Lev_t=\Lev_t-\Lev_{t-}>0, t\geq 0)=0$ and does not have non-increasing paths, i.e.~it is not the negative of a subordinator which is equivalent to $\mathrm{d}_X +\int_{-1}^0 |y| \Pi(dy)\in (0,\infty] $ in \eqref{LK}. As a byproduct, the L\'evy measure has support on the negative half-line, i.e.~$\Pi(0,\infty)=0$ in \eqref{LK}.  Then, this yields   that $\Psi$ admits an analytical extension to the lower half-plane, still denoted by $\Psi$, see \cite[Theorem 25.17]{Sato1999}, with $u \mapsto \Psi(-iu)$  a real-valued convex function on $\R^+$ with $\Psi(0)=0$ and $\lim_{u\to \infty} \Psi(-iu)= \infty$ and thus if $i\Psi'(0^+)<0$ then there exists an unique $\theta>0$ such that $\Psi(-i\theta)=0$. As $\Psi(-i u)$ is increasing and continuous for $[\theta_0, \infty)$ where $\theta_0=\theta \mathbb{I}_{\{i\Psi'(0^+)<0\}}$, it is a bijection and its inverse bijection $\phi: [0,\infty) \mapsto [\theta_0,\infty)$  satisfies  $\Psi(-i\phi(p))=p$ for $p\geq 0$.
It is known that $\phi$ is a Bernstein function, where we recall that $\phi$ is a Bernstein function if $\phi(u)-\phi(0)$ is of the form \eqref{LKs} with $\phi(0)\geq0$. %
In this context we  obtain the following refined results, where by a subordinator killed at rate $p\geq 0$, we mean
a $[0,\infty]$-valued c\`adl\`ag process which is in law equal to a subordinator  killed (i.e.~sent to the absorbing state $+\infty$) at time $\mathbf e_p$ with the understanding that $\mathbf e_0=\infty$.
\begin{prop}\label{cor:sn}
We have, for some (or equivalently any)  $a\geq 0$, that $\P(\T_a^{(\Sb)}<\infty)>0$ if and only if $\sigma^2>0$ or $\mathrm{d}_X +\int_{-1}^0 |y| \Pi(dy) - \dt \dk \in (0,\infty] $. Let us assume that this  holds.  %
Then,  the mapping $\Psi_{\Ikea}$ defined in \eqref{eq:Psidef} is the characteristic exponent of a spectrally  negative   L\'evy process and we denote by $\phi_{\Ikea}$ its inverse. Then, the following hold.

\begin{enumerate}[(1)]
	
\item \label{it:cor_1} The first passage time process $(\T_a^{(\Sb)})_{a\geq 0}$ is a subordinator killed at rate $\phi_{{\Sub_{\Sb}^{\triangleright 0}}}(0)$ with
\begin{equation}\label{eq:LTHT}
  \E\left[ e^{-q \T^{(\Sb)}_a } \mathbb I_{\{\T^{(\Sb)}_a <\infty\}} \right]= e^{-{\phi}_{\T^{(\Sb)}}(q)a}, \quad q\geq 0,
\end{equation}
where  $\phi_{\T^{(\Sb)}}(q)=\phi_{\Ikea}\circ \phi_{\Sub}(q)$ is a Bernstein function.

\item \label{it:csn2} If the inverse of the continuous and increasing function $q\mapsto \phi_{\T^{(\Sb)}}(q)=\phi_{\Ikea} \circ \phi_{\Sub}(q)$ is the Laplace exponent of a spectrally negative L\'evy process $\widetilde{\Lev}$ starting from $0$, then  we have the identity in law
\begin{equation}\label{eq:id_law}
      (\T^{(\Sb)}_a)_{a\geq 0} \stackrel{(d)}{=} (\widetilde{\Tm}_a)_{a\geq 0},
\end{equation}
where  $\widetilde{\Tm}_a = \inf\{t>0;\:  \widetilde{\Lev}_t > a\}$ is also the first passage time above $a$ of $\widetilde{\Lev}$. This condition holds for instance, if $\phi_{\Sb}\equiv 0$  (i.e.~$\Sb=0$ and thus $\T_a^{(\Sb)}=\T_a$), $\phi_{\Sub}(u)=u^{\alpha}$ and $\Psi(-iu)=u^{\beta}, u\geq 0,$ with $0<\alpha<1<\beta \leq 2\alpha$, see the example of Section \ref{ex:sel} below.

\item  \label{it:csn3} For any $a,q\geq0$ and $0< x\leq a$, we have, recalling that $\widehat{\T}_0 = \inf \{t>0;\: \X_t < 0\}$,
\begin{equation*}\label{eq:twosidedexit}
\begin{split}
\E_x \left[e^{-q \T_{a} }\mathbb{I}_{\{\T_{a} <\widehat{\T}_0\}} \right] = &  \frac{W^{(\phi_{\Sub}(q))}(x)}{W^{(\phi_{\Sub}(q))}(a)}, \\
\E_x \left[e^{-q \widehat{\T}_0 }\mathbb{I}_{\{\widehat{\T}_0<\T_{a}\}} \right] = & Z^{(\phi_{\Sub}(q))}(x) - \frac{ Z^{(\phi_{\Sub}(q))}(a)}{W^{(\phi_{\Sub}(q))}(a)}W^{(\phi_{\Sub}(q))}(x),
\end{split}
\end{equation*}
where  for $p\geq 0$ and $u>0$ large enough such that $\Psi(-iu)>p$,
\begin{equation*}
\int_{0}^{\infty} e^{-ux}W^{(p)}(x)dx=\frac1{\Psi(-iu)-p},
\end{equation*}
and for $p,x\geq 0$, $Z^{(p)}(x)=1+p\int_0^x W^{(p)}(y)d y$. %

\end{enumerate}
\end{prop}
\begin{rem}
Note that the function $\phi_{\T^{(\Sb)}}$ in item \eqref{it:cor_1} is not, in the case $\Sb \neq 0$, a simple composition of Bernstein functions. In fact, we could not find an analytical or direct proof of its Bernstein property. Instead, we resort to the theory of semi-regenerative processes to show that the process $(\T_a^{(\Sb)})_{a\geq 0}$ is a (possibly killed) subordinator with Laplace exponent $\phi_{\T^{(\Sb)}}$  from which the Bernstein  property follows. %
\end{rem}

\begin{rem}
We point out that, in item \eqref{it:csn3}, we  consider only the case $\Sb=0$ as our proof requires that the stopping times we consider should have their graph included in $\Rs$ a.s., which is the case for $\T^{(\Sb)}_a$ or $\widehat{\T}_a^{(\Sb)}$. However, excluding trivial cases, we have, for all $x>0$, that  $\P_x(\inf \{t>0;\: \X^{(\Sb)}_t < 0\}\notin\Rs\cup\{\infty\})>0$ for any $\Sb\neq 0$.
\end{rem}

\subsection{The fractional subdiffusive processes} \label{sec:self}

We proceed with the study of  another interesting  class of subdiffusive  processes  which are defined  by using the inverse of a stable  subordinator as the time change.  When the L\'evy process is a Brownian motion, it has been intensively  studied, see e.g.~\cite{Baeu,Hairer}, where it is called the fractional kinetic process.  In this spirit, we name such a generalization  a fractional subdiffusive process.  %
\begin{prop}\label{cor:ident}
Let us assume that $\Sub$ is an $\alpha$-stable subordinator  with $\phi_{\Sub}(u)=u^{\alpha}$,  $0<\alpha<1$.  Then, under   $\P_{-x}$, $x>0$, the following holds.
\begin{enumerate}[(1)]
  \item \label{it:cord1} We have
 \beq\label{idT}
\T_0 \stackrel{(d)}{=} \Sub_1 \times \Tm_0^{\frac{1}{\alpha}}
\eeq
where  $\times$ stands for the product of independent variables and we have set $\Tm_0 = \inf \{t>0;\: \Lev_t > 0\} $. Consequently, $\P_{-x}(\Tm_0=\infty)=\P_{-x}(\T_0=\infty)$ and otherwise on $ [0,\infty)$ the law of $\T_0$ is absolutely continuous with a density denoted by $f_{\T_0}$.
   \item \label{it:corm} Moreover, $\E_{-x}\left[\T_0^{\alpha}\right]=\infty $ but
    \begin{equation}\label{eq:mom}
    \E_{-x}\left[\T_0^{\delta}\right]<\infty \textrm{ for some } 0<\delta<\alpha \Leftrightarrow \int_0^{1}  \exp\left(\int_1^{\infty}e^{-qt}\P(\Lev_t\leq 0) \frac{dt}{t}\right)q^{-\frac{\delta}{\alpha}}dq < \infty.
    \end{equation}
    If one of these equivalent conditions holds, then,  \[ f_{\T_0} \textrm{ admits an analytic extension to the sector } \C_{\frac{1-\alpha}{\alpha}\frac{\pi}{2}}=\left\{z\in\C;\,|\arg z|<\frac{1-\alpha}{\alpha}\frac{\pi}{2}\right\}.\] Moreover, $ f_{\T_0}\in C^{\infty}_0(\R^+)$, the space of infinitely continuously differentiable functions vanishing at infinity,   and  its successive derivatives  $f^{(n)}_{\T_0}, n=0,1,\ldots,$  admit the Mellin Barnes representation,
  for any   $0<a <\min(\alpha,\frac{\delta}{\alpha})$,
\begin{equation*}\label{eq:MITd}
  	 f_{\T_0}^{(n)}(t)=\frac{(-1)^n}{2\pi i}\int^{a+i\infty}_{a-i\infty}t^{-z-n}\frac{\Gamma(z+n)}{\Gamma(z)}\frac{\Gamma(1-\frac{z}{\alpha})}{\Gamma(1-z)}\E_{-x}\left[\Tm_0^{\frac{z}{\alpha}}\right]dz,
  	 	\end{equation*}
  where the integral is absolutely convergent for any $ t\in \C_{\frac{1-\alpha}{\alpha}\frac{\pi}{2}}$.
\item If  $\E_{-x}\left[\Tm_0^{\delta+1}\right]<\infty \textrm{ for some  } \delta>0$, then, for any $n=0,1,\ldots$,
   \begin{equation}\label{eq:MITda}
  	 f_{\T_0}^{(n)}(t)\sim (-1)^n \frac{\sin(\alpha \pi)}{\pi}\Gamma(\alpha+n)\E_x\left[\Tm_0\right]t^{-\alpha -n } \textrm{ as } t\to \infty,	 	
  	 \end{equation}
  \end{enumerate}
  and thus $\P_{-x}(\T_0>t) \sim  \frac{\E_{-x}\left[\Tm_0\right]}{\Gamma(2-\alpha)}t^{1-\alpha} \textrm{ as } t\to \infty$.
\end{prop}
\begin{rem}
The factorization \eqref{idT} and the item \eqref{it:corm} reveal that the time-change smooths out the law  of the first passage  time $\T_0$ compared to the equivalent one for the L\'evy process. Indeed, first,  the law of $T_0$ is well known to do not be necessarily absolutely continuous (for instance when $\Lev$ is a compound Poisson process plus a positive drift), and the exponential decay along imaginary lines of the Mellin transform of $\Sub_1$, see the proof of this proposition, improves the regularity of the law of  $\T_0$ compared to the one  of $T_0$.
\end{rem}
\begin{rem}
Although the condition in \eqref{eq:mom} is not  completely explicit, one can get, using the discussion following \cite[Theorem 2]{Doney-Maller},  that $\E_{-x}\left[\T_0^{\delta}\right]<\infty$ for $0<\delta <\min(\alpha,\rho)$ where $\lim_{t\to \infty}\P(\Lev_t<0)=1-\rho\in [0,1]$.
\end{rem}

\section{Examples}\label{sec:ex}

In this section, we explore in detail two different examples that illustrate the main results on the first passage time problems. We emphasize  that one could generate a variety of additional examples as most of our results  are expressed in terms of the positive Wiener-Hopf factor ${\Phi}$ which admits an explicit expression in quite a lot of cases, e.g.~for any L\'evy process which has positive jumps of finite intensity with a rational Laplace transform, see \cite{Lewis-Mordecki-2008}.

\subsection{Self-similar subdiffusive processes} \label{ex:sel}

We proceed by investigating the case  when $\Lev$ is a $\mathfrak a$-stable process, $\mathfrak a \in (0,2)$ with positivity parameter  $\rho = \P(\Lev_1>0) \in (0,1)$ and $\Sub$ is, as in Proposition \ref{cor:ident}, an $\alpha$-stable subordinator, with $0<\alpha<1$. We exclude the case when  $X$ is a subordinator, that is  $\mathfrak a \in (0,1)$ and $\rho=1$ and recall that if $\mathfrak a \in (1,2)$ then $\mathfrak a \rho \geq \mathfrak a -1$. By \cite[Chapter VIII]{Bertoin-96}, we know that the L\'evy measure of $X$ is absolutely continuous  and takes the form
\begin{equation*}
\Pi(dy)=(c_+y^{-\mathfrak a-1}\mathbf{1}_{\{y>0\}}+c_{-}|y|^{-\mathfrak a-1}\mathbf{1}_{\{y<0\}})dy, \:y\in \R.
\end{equation*}
for some positive constants $c_{-},c_{+}>0$. %
Recall that  $\Lt$ has  a.s.~continuous and non-decreasing paths and inherits the $\frac{1}{\alpha}$-self-similarity property from $\Sub$ and thus as $\Lev$ and $\Lt$ are independent, one easily gets that $\X$ is a $\frac{\mathfrak a}{\alpha}$-self-similar process which means  that the identity
\begin{equation*}\label{eq:defselfsim}
  (\X_{c^{\frac{\mathfrak a}{\alpha}}t},\P_{cx})_{t\geq 0} \stackrel{(d)}{=} (c \X_t,\P_{x})_{t\geq 0}
\end{equation*}
holds in the sense of finite-dimensional distributions for any $c>0$ and $x\in \R$. We point out that this example is not treated in the recent paper \cite{LPS} where an in-depth analysis of the absorption time of a general class of self-similar non-Markovian processes is carried out.  In order to derive  the expression of the Mellin transform of $\widehat{T}_0=\inf \{t>0;\: \Lev_t < 0\} $, one uses the fact that $X$ killed upon entering the negative half-line is a positive self-similar Markov process of index $\mathfrak a$ and thus according to Lamperti, we have the identity in law
 \begin{equation}\label{eq:id_T_exp}
 \widehat{\Tm}_0 \stackrel{(d)}{=} x^{\mathfrak a} \int_{0}^{\infty}\exp(\mathfrak a Y_t)dt<\infty
 \end{equation}
 where $(Y_t)_{t\geq0}$ is a L\'evy process whose L\'evy-Khintchine exponent is expressed in terms of its Wiener-Hopf factors as follows
\begin{align}\label{eq:stable}
\Psi_{\mathfrak a Y} (z)&=-\frac{\Gamma(1+\mathfrak a z)}{\Gamma(1-\mathfrak{a}(1-\rho)+\mathfrak a  z)}\frac{\Gamma(\mathfrak{a}-\mathfrak a  z)}{\Gamma(\mathfrak{a}(1-\rho)-\mathfrak a z)}=-\phi_{\mathfrak a }^-(z)\phi_{\mathfrak a}^+(-z),
\end{align}
 see e.g.~\cite{Kypri}[Theorem 2.3] and \cite[Section 2]{LPS} where thereout  $\mathfrak a=\alpha$ for more details. Note that
 $ \P_{-x}(\inf \{t>0;\: \widehat{\Lev}_t > 0\} \leq t ) = \P_{x}(\widehat{\Tm}_0\leq t )$ where $\widehat{\Lev}=-X$, as the dual of $\Lev$, is also a $\mathfrak a$-stable process  with positivity parameter $1-\rho$.
 Then, recalling that
 the Barnes gamma function $G$ is the unique log-convex solution to the functional equation, for $u,\tau>0$,
 $G_\tau(u+1)=\Gamma\left(\frac{u}{\tau}\right)G_\tau(u)$
see \cite{Barnes}, we get that the mapping
\begin{equation}\label{eq:W_phm}
W^-_{\mathfrak a,\rho}(z+1)=\frac{G_{\frac{1}{\mathfrak a}}(1+\frac{1}{\mathfrak a}+\rho-1)}{G_{\frac{1}{\mathfrak a}}(\frac{1}{\mathfrak a}+1)}\frac{
G_{\frac{1}{\mathfrak a}}(z+\frac{1}{\mathfrak a}+1)}{G_{\frac{1}{\mathfrak a}}(z+\frac{1}{\mathfrak a}+\rho)}
 \end{equation}
$(\textrm{resp.~} W^+_{\mathfrak a, \rho}(z+1)= \frac{G_{\frac{1}{\mathfrak a}}\left(1-\rho\right)}{G_{\frac{1}{\mathfrak a}}\left(1\right)}
 \frac{G_{\frac{1}{\mathfrak a}}\left(z+2\right)}{G_{\frac{1}{\mathfrak a}}\left(z+2-\rho\right)})$ is the unique log-convex on $\R^+$ solution   to $f(z+1)=\phi_{\mathfrak a }^-(z)f(z), \Re(z)>1-\rho-\frac{1}{\mathfrak a}$ (resp.~$f(z+1)=\phi_{\mathfrak a }^+(z)f(z), \Re(z)>\rho-1), f(1)=1$, see again \cite{Patie-Savov-BG} for a study of these functional equations for general Bernstein functions.
 Next, using the identity  \eqref{eq:id_T_exp} and the expression of the Mellin transform  of the so-called exponential functional which is found in \cite[Theorem 2.4]{Patie-Savov-BG} (note that the exponential functional in this paper is defined with the L\'evy process $\xi=-\alpha Y$) we get in this case that, for any $-1<\Re(z)<1-\rho$,
\[ \E_x\left[\widehat{\Tm}_0^{z}\right] = x^{\mathfrak a z}\frac{\Gamma(\mathfrak{a} )}{\Gamma(\mathfrak{a}(1-\rho))}\frac{\Gamma(z+1)W^+_{\mathfrak a,\rho}(-z)}{W^-_{\mathfrak a,\rho}(z+1)}.\]
By means of the identity in law \eqref{idT} and the expression of the (shifted) Mellin transform of the stable subordinator recalled in \eqref{eq:ms}, we get that, for any $-\alpha<\Re(z)< \alpha(1-\rho)$,
\begin{equation}
\E_x\left[\widehat{\T}_0^{z}\right] = x^{\frac{\mathfrak a}{\alpha} z}\frac{\Gamma(\mathfrak{a} )}{\Gamma(\mathfrak{a}(1-\rho))}\frac{\Gamma(1-\frac{z}{\alpha})}{\Gamma(1-z)}\frac{\Gamma(\frac{z}{\alpha}+1)\Gamma(-\frac{z-1}{\alpha})W^+_{\mathfrak a,\rho}(-\frac{z}{\alpha})}{\Gamma(\frac{z+1}{\alpha})W^-_{\mathfrak a,\rho}(\frac{z}{\alpha}+1)}
\end{equation}
where we recall that $ \widehat{\T}_0=\inf \{t>0;\: \X_t < 0\}$.
Moreover, %
applying the (complex) Stirling formula for the gamma function recalled in \eqref{eq:asympt_gamma} below together with the asymptotic expansion of the Barnes gamma functions found in \cite[formula (4.5)]{BK} to both $W^+_{\mathfrak a, \rho}$ and $W^-_{\mathfrak a, \rho}$  one gets that, for any fixed $ -\alpha<a< \alpha(1-\rho)$ and with $z=a+ib$,
\[ \left|\E_x\left[\widehat{\T}_0^{z}\right]\right| \leq  C_a e^{-|b|(2-\alpha+\mathfrak{a}(2\rho-1))\frac{\pi}{2\alpha}}\]
for some $C_a>0$ and where $2-\alpha+\mathfrak{a}(2\rho-1)>0$. This combined with the analyticity property of $z\mapsto \E_x\left[\widehat{\T}_0^{z}\right]$ on the strip $-\alpha<\Re(z)<\alpha(1-\rho)$ allows us to use Mellin inversion techniques combined with a dominated convergence argument to get that the law of $\widehat{\T}_0$ is absolutely continuous with a density  $f_{\widehat{\T}_0} \in C^{\infty}_0(\R^+)$ and which admits an analytical extension to  the sector $\{z\in\C;\,|\arg z|<(\frac{1-\alpha}{\alpha}+\frac{1}{\mathfrak{a}}+(2\rho-1))\frac{\pi}{2}\}$  given, along with its successive derivatives,  by the following Mellin Barnes integral   for any $n \in \mathbb{N}$
  and $ -\alpha<a< \alpha(1-\rho)$,
  \begin{equation}\label{eq:MITd}
  	 f_{\widehat{\T}_0}^{(n)}(t)=(-1)^n\frac{\Gamma(\mathfrak{a} )}{\Gamma(\mathfrak{a}(1-\rho))}\frac{1}{2\pi i}\int^{a+i\infty}_{a-i\infty}t^{-z-n}\Gamma(z+n)\frac{\Gamma(1-\frac{z}{\alpha})}{\Gamma(1-z)}\frac{W^+_{\mathfrak a,\rho}(1-z)}{W^-_{\mathfrak a,\rho}(z)}dz,
  	 	\end{equation}
  where the integral is absolutely convergent for any $t>0$, see e.g.~\cite[Section 1.7.4]{Patie-Savov-16} for a review of some basic facts on Mellin transform.

  Let us now focus to the case when $X$ is  spectrally negative that is its L\'evy measure takes the form $\Pi(dy)=c_{-}|y|^{-\mathfrak a-1}\mathbf{1}_{\{y<0\}}dy$. In this case,  $\rho=\frac{1}{\mathfrak a}$ with $1<\mathfrak a\leq 2$ and   $\Psi(-iu)=u^{\mathfrak a}, u\geq 0$,  where the case $\mathfrak a=2$ corresponds to a (scaled) Brownian motion.  It is easy to check, from Proposition \ref{cor:sn}\eqref{it:cor_1}, that  under $\P$, the process of first passage times $(\T_a)_{a\geq 0}$, where we recall that $\T_a=\T_a^0$, is a $\frac{\alpha}{\mathfrak a}$-stable  subordinator which is, in the case when $\mathfrak a \leq 2 \alpha$, the law of the process of first passage times of the spectrally negative stable L\'evy process with index $1<\frac{\mathfrak a}{\alpha}\leq 2$. Next, the recurrence relation of the Barnes function given above and  the identity $G_{\tau}(z+\tau)=(2\pi)^{\frac{\tau-1}{2}}\tau^{-z+\frac12}\Gamma(z)G_{\tau}(z)$ found in \cite[bottom of p.371]{Barnes}, give from \eqref{eq:stable} that
for any $-1<\Re(z)<1-\frac{1}{\mathfrak a}$,
\begin{equation*}
\E_x\left[\widehat{\T}_0^{z}\right] = x^{\frac{\mathfrak a}{\alpha} z} \frac{\sin(\pi/\mathfrak a)}{\pi}\frac{\Gamma(1-\frac{z}{\alpha})}{\Gamma(1-z)}\frac{\Gamma(1 + \frac{z}{\alpha}) \Gamma(\frac{z}{\alpha} + \frac{1}{\mathfrak a})\Gamma(1- \frac{1}{\mathfrak a}-\frac{z}{\alpha})}{
\Gamma(1 + \frac{\mathfrak a z}{\alpha}) }.
\end{equation*}
Then, a classical application of the Cauchy  theorem to the Barnes integral of the form \eqref{eq:MITd} which is obtained by Mellin inversion yields that the density of  $\widehat{\T}_0$ when $\X$ is starting from $1$ has the series representation, for any $t>0$,
  \begin{equation*}
f_{\widehat{\T}_0}(t) = \frac{1}{ \mathfrak a \pi}\sum_{n=0}^\infty (\sin(\alpha \pi(n+1))a_n +\sin(\alpha \pi(n+1-\frac{1}{\mathfrak a})) b_nt^{ \frac{\alpha}{\mathfrak a}})t^{-\alpha(n+1)-1}
\end{equation*}
where the coefficients $a_n=-\frac{\Gamma(\alpha(n+1))}{\Gamma(\mathfrak a(n+1))} $
and $b_n= -\frac{\Gamma(\alpha(n+1-\frac{1}{\mathfrak a}))}{\Gamma(\mathfrak a(n+1-\frac{1}{\mathfrak a}))} $ define two entire power series.

\subsection{A Cox-renewal process}\label{sec:cox}

We assume here  that $X$ is a compound Poisson process with intensity $\lambda>0$ and positive exponentially distributed jumps with parameter $p>0$, i.e.~$X$ is a subordinator without drift, i.e.~$\mathrm{d}_X=0$, and L\'evy measure $\Pi(d y) = \lambda p e^{-p y}\mathbf 1_{\{y>0\}} d y$, which implies that $\Psi(z)=\frac{i\lambda z}{p-iz}$, $\Im(z)>-p$.
As indicated in Section \ref{sec_intro}, a Poisson process time-changed by $\ell$ is a renewal process and so $\X$ is a compound renewal process with positive exponentially distributed jumps. If $\Sb$ is just a drift, i.e.~$\Sb_t=\dt t$ with $\dt>0$, then $\T_a^{(\Sb)}$ corresponds to the ruin time in the so-called renewal risk or Sparre-Andersen model with exponentially distributed claims, initial capital $a$, premium rate $\dt$ and interarrival distribution whose Laplace transform is given by \eqref{assump_hold} with $\phi=\phi_\Sub$.
With $\Sb$ a general non-zero subordinator, $\T_a^{(\Sb)}$ can be seen as the ruin time of a risk process where the size and arrival of the claims are as in the aforementioned renewal risk model but the inflow of capital is more general than a deterministic premium rate as it includes random capital injections that are modelled by the jumps of $\Sb$.
Within this setting we have that $\Psi_{\Ikea}$ given by \eqref{eq:Psidef} is the characteristic exponent of a L\'evy process with non-monotone sample paths whose L\'evy measure coincides with $\Pi$ on the positive half-line and so we can use  \cite[Theorem 2.2]{Lewis-Mordecki-2008} to get,
\begin{equation}\label{WH_renewal_example}
\Phi_\Ikea(\varrho;i u) = \frac{u+p}{p} \frac{ R(\varrho)}{u+ R(\varrho)}, \quad \varrho>0 \ \text{and} \ u,q\geq 0,
\end{equation}
where $ R(\varrho)$ is the unique positive solution to $\Psi_\Ikea(-i R(\varrho)) =\varrho$, see \cite[Lemma 1.1]{Lewis-Mordecki-2008}). Hence by Theorem \ref{thm2}, for $u,v,q >0$ with $u\neq v$,
\begin{equation*}
\begin{split}
\int_0^\infty e^{-u a } \E\left[ e^{ -q\T_a^{(\Sb)} - v ( \X^{(\Sb)}_{\T_a^{(\Sb)}} -a ) } \right] d a = &
\frac{1}{u-v} \left( 1 -  \frac{u+p}{u+ R(\phi_\Sub(q))}  \frac{v+ R(\phi_\Sub(q))}{v+p}    \right) \\
= &   \frac{p- R(\phi_\Sub(q))}{p+v} \frac1{u+ R(\phi_\Sub(q))}.
\end{split}
\end{equation*}
By Laplace inversion, we get for $a\geq 0$, $q>0$ and $y>0$,
\begin{equation}\label{example_P-R}
\E\left[ e^{ -q\T_a^{(\Sb)}  } \mathbb{I}_{\{   \X^{(\Sb)}_{\T_a^{(\Sb)}} -a  \in   d y\}} \right]  =  \frac{p- R(\phi_\Sub(q))}{p} e^{- R(\phi_\Sub(q)) a} p e^{-p y} d y.
\end{equation}
We see that, for any $a\geq 0$, the overshoot $\X^{(\Sb)}_{\T_a^{(\Sb)}} -a$ %
  is exponentially distributed with parameter $p$ and independent of the first passage time $\T_a^{(\Sb)}$, a property which was known for  the corresponding  L\'evy  process with positive exponential jumps and carries over for the time-changed version. Note that the L\'evy process  $\Xbs$ as defined in \eqref{levy_XbK} below is the one with characteristic exponent  $\Psi_{{\Sub_{\Sb}^{\triangleright 0}}}$. Since $\mathbb E[\Xbs_1]=\mathbb E[X_1]-\mathbb E[\Sb_{\Sub_t}]=\frac\lambda p-\phi'_\Sub(0)\phi'_\Sb(0)\in[-\infty,\infty)$, the L\'evy process $\Xbs$ drifts to $-\infty$  if and only if $\phi'_\Sub(0)\phi'_\Sb(0)>\frac\lambda  p$, see  \cite[theorems 36.5 and 36.6]{Sato1999}. If this is the case, then \eqref{WH_renewal_example} also holds for $p=q=0$ and $u\geq 0$, see again \cite{Lewis-Mordecki-2008}, and so by Theorem \ref{thm2} the ruin probability is given by
\begin{equation}\label{example_ruinprob}
\mathbb P(\T_a^{(\Sb)}<\infty) = \frac{ p- R(0)}{ p} e^{- R(0) a}, \quad a\geq 0.
\end{equation}
To make the link with some of the existing literature, we consider the so-called fractional Poisson risk model with exponential claims for which $\phi_\Sb(u)=\dt u,  \dt>0$ and $\phi_\Sub(u)=u^\alpha, \alpha\in(0,1)$, which implies that the inter-arrival distribution is given by \eqref{mittag-leffler_distr}. In that case $\Xbs$ always drifts to $-\infty$ as $\phi_\Sub'(0)=\infty$ and for any $q\geq 0$, $R_q:= R(\phi_\Sub(q))$ is the unique positive solution to
\begin{equation*}
\frac{\lambda R_q}{ p-R_q}   - (\dt R_q+q)^\alpha  = 0.
\end{equation*}
We see that \eqref{example_P-R} and \eqref{example_ruinprob} are consistent with  \cite[Propositions 2 and 3]{Biard-Saussereau-2014} and  \cite[Example 4.3]{rockpaper}, though note there is a typo  in \cite[ Proposition 2]{Biard-Saussereau-2014}, namely the left-hand side of the main identity there equals one minus the right-hand side.

\section{Proofs} \label{sec:proof}

In the proofs below $\mathcal B(\mathbb R^p)$ denotes  the Borel $\sigma$-algebra on $\mathbb R^p, p\in \mathbb N$.
We recall %
that a filtration  $(\mathcal G_t)_{t\geq 0}$ is called right-continuous if $\cap_{\epsilon>0}\mathcal G_{t+\epsilon}=\mathcal G_t$ for all $t\geq 0$.

\subsection{Proof of Proposition   \ref{prop:regenerative}}\label{sec:proofprop1}
As announced after Proposition \ref{prop:regenerative}, we state and proof the following more general version, where $\Sub$  still satisfies the condition \ref{eq:ass}, that is it has increasing sample paths.
\begin{prop}\label{prop:reg}
 Proposition \ref{prop:regenerative} holds with $\X$ in   \eqref{def:reglevy}  being defined from  a two-dimensional L\'evy process  $(\Lev,\Sub)$ with possibly dependent components.
 \end{prop}
\begin{proof}
	Since $(\Lev,\Sub)$ is a L\'evy process, the filtration $(\mathcal F^\P_t)_{t\geq 0}$ is right-continuous, see e.g.~\cite[Proposition I.4]{Bertoin-96}. Moreover, the fact that  $\Sub$ is adapted to $(\mathcal F^\P_t)_{t\geq 0}$ entails that for each $t\geq 0$, $\Lt_t$ is a stopping time with respect to $(\mathcal F^\P_t)_{t\geq 0}$, see e.g.~\cite[Lemmas 7.6 and 7.2]{Kallenberg2001}. Hence for any $t\geq 0$, $\widetilde{\mathcal{F}}_t=\mathcal F^\P_{\Lt_t}$ is well-defined. As $\Lt$ is non-decreasing, it is easy to see that $(\widetilde{\mathcal{F}}_t)_{t\geq 0}$ is a filtration. Since $\Lev$ is adapted to $(\mathcal F_t)_{t\geq 0}$ and has right-continuous sample paths, for each $t\geq 0$, $\X_t=\Lev_{\Lt_t}$ is $\widetilde{\mathcal{F}}_t$-measurable, see e.g.~\cite[p.~122]{Kallenberg2001}. Using again that $\Lt$ is non-decreasing and continuous, we observe that for any $A\in\cap_{\epsilon>0}\widetilde{\mathcal{F}}_{t+\epsilon}$ and $s\geq 0$,
	$$A\cap\{\ell_t\leq s\}= \bigcap_{n\geq 1}\bigcup_{m\geq 1} A\cap \{\ell_{t+1/m}\leq s+1/n\} \in\bigcap_{n\geq 1}\mathcal F^\P_{s+1/n}=\mathcal F^\P_s,$$
	which shows that $A\in\widetilde{\mathcal{F}}_t$ and thus the filtration $(\widetilde{\mathcal{F}}_t)_{t\geq 0}$ is right-continuous.
		In order to prove the second statement, let $T$ be an $(\widetilde{\mathcal{F}}_t)_{t\geq 0}$-stopping time  such that its graph $[T] \subseteq \Rs$ a.s.~and satisfying $\mathbb P(T<\infty)>0$.
	We first show that $\Lt_{T}$ is also an $(\mathcal F^\P_t)_{t\geq 0}$-stopping time. If $T$ has a countable or finite range, say $\{s_1,s_2,\ldots\}$, then, for any $t\geq 0$, $\{\Lt_T\leq t\}=\cup_{k\geq 1}\{T=s_k\}\cap \{\Lt_{s_k}\leq t\}\in\mathcal F^\P_t$ since $T$ is an $(\widetilde{\mathcal{F}}_t)_{t\geq 0}$-stopping time and thus  $\Lt_{T}$ is an $(\mathcal F^\P_t)_{t\geq 0}$-stopping time. The general case then follows by an approximation of $T$ by a sequence of non-increasing stopping times with a countable or finite range combined with the facts that $\ell$ is non-decreasing and both $\Lt$ and $(\mathcal F^\P_t)_{t\geq 0}$ are  right-continuous, see \cite[Lemmas 7.3(ii) and 7.4]{Kallenberg2001}. By the same arguments one can show that $\mathcal F^\P_{\Lt_T}\supset\widetilde{\mathcal{F}}_T$.
		Next, since  $(\Lev,\Sub)$ is a L\'evy process it follows, for any $t\geq 0$, that the process $(\Lev_{t+s}-\Lev_t,\Sub_{t+s}-\Sub_t)_{s\geq 0}$ is independent of $\mathcal F_t$. Since $\mathcal N$ consists of null-events the aforementioned process is also independent of $\mathcal F_t\cup \mathcal N$ and since $\mathcal F_t\cup \mathcal N$ is closed under finite intersections, it follows  by a monotone class argument  that  $(\Lev_{t+s}-\Lev_t,\Sub_{t+s}-\Sub_t)_{s\geq 0}$ is independent of $\mathcal F^\P_t$. %
 Hence,  $(\Lev,\Sub)$  is also a (two-dimensional) L\'evy process with respect to $(\mathcal F^\P_t)_{t\geq 0}$.
		Now, let $n\geq 1$ and $t_i,r_i,a_i\geq 0$, $B_i\in\mathcal B(\mathbb R)$, $i=1,\ldots,n$ and  $H\in\mathcal F^\P_{\Lt_T}$.
	Then, writing $\mathbf{B}_n=\cap_{i=1}^n   \{\Lev_{r_i+\Lt_T}-\Lev_{\Lt_T}\in B_i, \Lt_{t_i + T}-\Lt_T>a_i \} \cap H$ and simply $\mathbb P^T(\cdot)=\mathbb P(\cdot|T<\infty)$, we have
	\begin{eqnarray}\label{indep_X_L_H}
	\mathbb P^T (  \mathbf{B}_n) &  = & \mathbb P^T  \left( \cap_{i=1}^n \{\Lev_{r_i+\Lt_T}-\Lev_{\Lt_T}\in B_i, \Sub_{a_i+\Lt_T} < T + t_i\} \cap H  \right) \nonumber \\
	& = & \mathbb P^T  \left( \cap_{i=1}^n \{\Lev_{r_i+\Lt_T}-\Lev_{\Lt_T}\in B_i, \Sub_{a_i+\Lt_T} -\Sub_{\Lt_T}< t_i\} \cap H \right) \nonumber  \\
	& = & \mathbb P  \left( \cap_{i=1}^n \{\Lev_{r_i}\in B_i, \Sub_{a_i}< t_i\} \right) \mathbb P^T(H)\nonumber  \\
	& = & \mathbb P  \left( \cap_{i=1}^n \{\Lev_{r_i}\in B_i, \Lt_{t_i}> a_i\} \right) \mathbb P^T(H),
	\end{eqnarray}
	where we used for the first and last equality  the fact that $\Sub$ is increasing to get that $\{\Lt_a>t\}=\{\Sub_t<a\}$ for any $a,t\geq 0$, for the second one that $\mathbb P(T\in\Rs\cup\{\infty\})=1$, whereas for the third one the strong Markov property for L\'evy processes (see e.g.~\cite[Proposition I.6]{Bertoin-96}) which  can be applied as  $(\Lev,\Sub)$  is a L\'evy process with respect to $(\mathcal F^\P_t)_{t\geq 0}$, $\Lt_T$ is an $(\mathcal F^\P_t)_{t\geq 0}$-stopping time and $\{T<\infty\}=\{\Lt_T<\infty\}$. %
	Equation \eqref{indep_X_L_H} with $H=\Omega$ shows that $(\Lev_{t+\Lt_T}-\Lev_{\Lt_T}, \Lt_{t+T}-\Lt_T)_{t\geq 0}$ under $\mathbb P^T$ has the same finite-dimensional distributions, and thus law, as $(\Lev,\Lt)$ under $\mathbb P$. Then using this established fact in \eqref{indep_X_L_H}, we further conclude that $(\Lev_{t+\Lt_T}-\Lev_{\Lt_T}, \Lt_{t+T}-\Lt_T)_{t\geq 0}$ and the $\sigma$-algebra $\mathcal F^\P_{\Lt_T}$ are independent under $\mathbb P^T$.
	Consequently, writing for some $n\geq 1, \mathds B_n=\cap_{i=1}^n \{\X_{t_i + T}-\X_T\in B_i\} \cap H$, we get
	\begin{eqnarray*}%
	\mathbb P^T (\mathds B_n ) \nonumber& = &
	 \int_{[0,\infty)^{n}} \mathbb P^T \left( \cap_{i=1}^n \{\Lev_{\Lt_T+r_i}-\Lev_{\Lt_T}\in B_i\}\cap H | \cap_{i=1}^n \{ \Lt_{t_i + T}-\Lt_T=r_i \} \right)\nonumber \\
	& & \times \:  \mathbb P^T( \cap_{i=1}^n \{ \Lt_{t_i + T}-\Lt_T\in d r_i \}   ) \nonumber \\
	& = &  \int_{[0,\infty)^n} \mathbb P \left( \cap_{i=1}^n \{\Lev_{r_i}\in B_i\} | \cap_{i=1}^n\{ \Lt_{t_i} = r_i \}
	\right)  \mathbb P^T(H) \mathbb P( \cap_{i=1}^n\{ \Lt_{t_i}\in d r_i \}) \nonumber\\
	& = & \mathbb P \left( \cap_{i=1}^n\{\X_{t_i}\in B_i\} \right) \mathbb P^T(H).
	\end{eqnarray*}
	From this equation we can conclude that under $\mathbb P^T$ the process $(\X_{T+t}-\X_T)_{t\geq 0}$ is independent of $\mathcal F^\P_{\Lt_T}\supset\widetilde{\mathcal{F}}_T$  and has the same law as $\X$ under $\mathbb P$. %
Finally, observe, since $\ell$ is continuous, that $\Rs=\{t\geq 0; \: \Sub_{\ell_t}=t  \}=\{t\geq 0; \: \Sub_y=t \ \text{for some $y\geq 0$}\}$, that is $\Rs$ is the range of the subordinator $\Sub$ and thus it is a regeneration set in the sense of \cite[Chap.~2]{Maison}.  Next, let  $T$ be an $(\widetilde{\mathcal{F}}_t)_{t\geq 0}$-stopping time  such that  $[T] \subseteq \Rs$,   $G$ a positive and $\mathcal{F}^0=\sigma(\X_t, t\geq 0)$-measurable function %
then, for any $t\geq0$,
\begin{eqnarray*}
\E\left[ G(\X_{t+T}) | \widetilde{\mathcal{F}}_T\right]& = & \E\left[ G(\X_{t+T}-\X_{T}+\X_{T}) | \widetilde{\mathcal{F}}_T\right] \\& = & \int_{\R}\E\left[ G(\X_{t+T}-\X_{T}+y) | \widetilde{\mathcal{F}}_T\right] \P(X_T\in dy) \\& = & \int_{\R}\E\left[ G(\X_{t}+y)\right] \P(X_T\in dy) \quad \P\textrm{-a.s. on } \{T<\infty\} \\& = & \E_{\mathds X_{T}}\left[ G(\X_{t})\right] \quad \P\textrm{-a.s. on } \{T<\infty\},
\end{eqnarray*}
where we used for the second identity that $X_T$ is $\widetilde{\mathcal{F}}_T$-measurable, for the third one the regenerative property of L\'evy type \eqref{eq:SMPL} of $\X$ and finally that $\E\left[ G(\X_{t}+y)\right]=\E_y\left[ G(\X_{t})\right]$, which completes the proof.
\end{proof}

\subsection{Proof of Theorem  \ref{thm1}} \label{sec:proofmain}
The proof of  Theorem  \ref{thm1} is split into several intermediate steps which may be of independent interests. The plan behind the proof is to reduce, via the time change and a change of measure, the problem of characterizing the distribution of the first passage time $\T_a^{(\Sb)}$ of $\X$ over the stochastic boundary $a+\Sb_t$ to the one involving only first passage times of L\'evy processes over the constant boundary $a$.
We remark that the proof is substantially easier in the special case where the boundary $a+\Sb_t$ is a constant, i.e.~$\Sb=0$. Indeed, in that case both the second half of Lemma \ref{lem:T=T} as well as Lemma \ref{lem_levypair} below are trivial (whereas the first half of Lemma \ref{lem:T=T} is not needed), Lemma \ref{lem:T=kT} is somewhat easier to prove and for the final step one can use the independence of the processes involved instead of resorting to a change of measure. We emphasize that the case $\Sb\neq 0$ does not readily follow  from the case $\Sb=0$  since the process $\X^{(\Sb)}$, which recall is defined by $\X^{(\Sb)}_t=\X_t-\Sb_t$, is not in general a L\'evy process time-changed by $\ell$. Though we will show that we can study $\T_a^{(\Sb)}$ by considering a particular time-changed L\'evy process, it  is interesting to note that the case $\Sb\neq0$ creates a dependence  between the time change process and the underlying L\'evy process, provided we are in the non-trivial case where the subordinator $\Sub$ has jumps.
We also point out  that, in lemmas \ref{lem:T=T} and \ref{lem:T=kT}, the fact that the L\'evy process $\Sb$ has non-decreasing sample paths is crucial.

We start the proof by introducing two additional processes, namely %
the process  $\Xbs=(\Xbs_t)_{t\geq 0}$ defined by
\begin{equation}\label{levy_XbK}
\Xbs_t %
=\Lev_t-\Sb_{\Sub_t}, \quad t\geq 0,
\end{equation}
which is easily seen to be  a L\'evy process with characteristic exponent $\Psi_{{\Sub_{\Sb}^{\triangleright 0}}}$ (see Section 30 of \cite{Sato1999} or Lemma \ref{lem_levypair} below)
 and the time-changed L\'evy process $\X^{(\Sb_{\Sub_\Lt})}= ( \X^{(\Sb_{\Sub_\Lt})}_t )_{t\geq 0}$ defined by
 \begin{equation*}
 \X^{(\Sb_{\Sub_\Lt})}_t =  \Xbs_{\Lt_t} %
 =\X_{t}-\Sb_{\Sub_{\Lt_t}}. %
 \end{equation*}
Further, for any process $Z=(Z_t)_{t\geq 0}$, we denote, from now on, by $\T_a(Z)$  its first passage time above the level $a$, that is
 \begin{equation*}
 \T_a(Z) = \inf\{t>0; \: Z_t>a\}.
 \end{equation*}
 Observe that $\T_a^{(\Sb)}=\ST{\T}{\Xb}$. %
 When $\Sb\neq 0$ and the subordinator $\Sub$ has jumps, the processes $\Xb$ and $\X^{(\Sb_{\Sub_\Lt})}$ are not equal (unless $\Sub$ has no jumps), but clearly we do have that $\Xb_T=\X^{(\Sb_{\Sub_\Lt})}_T$ for any random variable $T$ such that $[T] \subseteq \Rs$ where we recall  that $\Rs$ is the regeneration set  defined in \eqref{regen_set}.
 The next lemma shows that $[\ST{\T}{\Xb}] \subseteq \Rs$, provided $\ST{\T}{\Xb}<\infty$, which implies that $\ST{\T}{\Xb}=\ST{\T}{\X^{(\Sb_{\Sub_\Lt})}}$ $\P$-a.s.~and so we can reduce the problem of characterizing the law of $\ST{\T}{\Xb}$ to that of the first passage time of a (dependently) time-changed L\'evy process.

\begin{lem}\label{lem:T=T}
We have, for any $a\geq 0$, $[\ST{\T}{\Xb}] \subseteq \Rs$ $\P$-a.s.~and consequently $\ST{\T}{\Xb}=\ST{\T}{\X^{(\Sb_{\Sub_\Lt})}}$ $\P$-a.s.
\end{lem}
\begin{proof}
	Note that  throughout  this proof all identities/inequalities are in the $\P$-a.s.~sense and for sake of clarity we shall not mention it.  Since $\Sub_{\Lt_t}\geq t$ and $\Sb$ has non-decreasing sample paths we have
	$\Xb_t - \X^{(\Sb_{\Sub_\Lt})}_t = \Sb_{\Sub_{\Lt_t}} - \Sb_t \geq 0$. This implies that %
	$\ST{\T}{\Xb}\leq \ST{\T}{\X^{(\Sb_{\Sub_\Lt})}}$ and we can therefore assume,  without loss of generality, %
	that $\ST{\T}{\Xb}<\infty$.
		We next prove that $[\ST{\T}{\Xb}] \subseteq  \Rs$.
 To this end, let us introduce the process $\Sub^{-}_\Lt=(\Sub^{-}_{\Lt_t}=\Sub_{\Lt_t}-t)_{t\geq 0}$. According to \cite[Exercise IV.6.2]{Bertoin-96}, $\Sub^{-}_\Lt$ is a recurrent Markov process %
 with $0$ as a recurrent and regular point and  $\Lt$ is, up to a normalizing constant, its local time at $0$. Its zero set is equal to $\Rs$ and coincides with the range of $\Sub$, i.e.~$\Rs=\{t\geq 0; \: \Sub_y=t \ \text{for some $y\geq 0$}\}$.  %
	Let $(\g,\d)$ with $0<\g<\d$ be an excursion interval by which we mean that $\Sub^{-}_\Lt>0$ for all $t\in(\g,\d)$ and there exists no larger open interval on which $\Sub^-_\Lt$ is (strictly) positive. Note that the number of excursion intervals is countable, the union of all excursion intervals equals $[0,\infty)\setminus\overline{\Rs}$, where $\overline{\Rs}$ denotes the closure of $\Rs$ and $\overline{\Rs}\setminus{\Rs}$ consists precisely of the left-endpoints of the excursion intervals or equivalently the jump times of $\Sub_\Lt=(\Sub_{\Lt_t})_{t\geq 0}$, see e.g.~\cite[Section 1.4]{Bertoin_sub_lec}.
	Since $\Sb$ has non-decreasing sample paths and $\Lt$ is constant on $(\g,\d)$ and thus on $[\g,\d]$ by continuity, we have $\Xb_t\leq \Xb_\g$ for any $t\in[\g,\d]$. Hence $\mathbb P(\ST{\T}{\Xb}\in (\g,\d])=0$, which implies that $[\ST{\T}{\Xb}] \subseteq \overline{\Rs}$. Further, denoting $\Delta Z_t=Z_t-\lim_{s\uparrow t}Z_s$, we have
	\begin{eqnarray*}
		\mathbb P(\ST{\T}{\Xb} = \g)
		&=&\mathbb P(\ST{\T}{\Xb} =\g, \Delta \X_\g = 0) \\
		&=&\mathbb P(\ST{\T}{\Xb} =\g,\Delta \Xb_\g \leq 0)\\
		&=&\mathbb P(\ST{\T}{\Xb} =\g, \Xb_{\g}=a) \\
		&=&0,
	\end{eqnarray*}
where (i) the first equality follows because $\g$ is a jump time of $\Sub_\Lt$ and thus $\Lt_\g$ is a jump time of $\Sub$ by continuity of $\Lt$ and so since $\Sub$ and $\Lev$ jump %
at different times by independence, $\Lt_\g$ is not a jump time of $\Lev$ and thus $\g$ is not a jump time of $\X$ by continuity of $\Lt$, (ii) the second equality follows  because  $\Delta \Sb_\g\leq 0$ and (iii) the last equality is due to $\Xb_t\leq \Xb_\g$ for $t\in[\g,\d]$. We conclude that $[\ST{\T}{\Xb}] \subseteq \Rs$.
	As  $\ST{\T}{\Xb}=\ST{\T}{\X^{(\Sb_{\Sub_\Lt})}}$ for all $a>0$ implies the equality for $a=0$ by right-continuity of $\ST{\T}{\Xb}$ and $\ST{\T}{\X^{(\Sb_{\Sub_\Lt})}}$ in $a$, we can assume without loss of generality that $a>0$.
Now consider the event $A=\{\ST{\T}{\Xb} < \ST{\T}{\X^{(\Sb_{\Sub_\Lt})}}\}$. As we already showed $\ST{\T}{\Xb}\leq \ST{\T}{\X^{(\Sb_{\Sub_\Lt})}}$, the proof is finished once we show that $\mathbb P(A)=0$.
As $[\ST{\T}{\Xb}] \subseteq \Rs$, we have $\Xb_{\ST{\T}{\Xb}}= \X^{(\Sb_{\Sub_\Lt})}_{\ST{\T}{\Xb}}$  %
and thus
we have on $A$, $\Xb_{\ST{\T}{\Xb}}=\X^{(\Sb_{\Sub_\Lt})}_{\ST{\T}{\Xb}}=a$ and there exists $\epsilon>0$ such that $\X^{(\Sb_{\Sub_\Lt})}_{s+ \ST{\T}{\Xb}}\leq a$ for all $s\in[0,\epsilon]$. If the L\'evy measure of $\Sub$ has finite mass, then $\Rs$ consists of intervals that are open from the right and thus there exists $0<\delta<\epsilon$ such that $\Xb_{s+ \ST{\T}{\Xb} }=\X^{(\Sb_{\Sub_\Lt})}_{s+ \ST{\T}{\Xb}}\leq a$ for all $s\in[0,\delta]$ on $A$. By definition of $\ST{\T}{\Xb}$ this cannot happen  and thus $\mathbb P(A)=0$. Hence we assume now without loss of generality that the L\'evy measure of $\Sub$ has infinite mass.	 	
	 	 As $[\ST{\T}{\Xb}] \subseteq \Rs$ %
and $\Rs$ does not have isolated points from the right (since $0$ is regular for the Markov process $\Sub^{-}_\Lt$ as $\Sub$ is not a compound Poisson process by assumption), there exists $s'\in(0,\epsilon]$ such that $[s'+\ST{\T}{\Xb}] \subseteq \Rs$, which implies, by definition of $\Rs$, that $\Lt_{\ST{\T}{\Xb}}<\Lt_{s'+\ST{\T}{\Xb}}$. Hence, since $\X^{(\Sb_{\Sub_\Lt})}_{s+ \ST{\T}{\Xb}}\leq a$ for all $s\in[0,\epsilon]$ and $\Lt$ has continuous and non-decreasing sample paths, $\Xbs_{{t+\Lt_{\ST{\T}{\Xb}}}} \leq a$ for all $t\in \left[0,\Lt_{s'+\ST{\T}{\Xb}}-\Lt_{\ST{\T}{\Xb}}\right]$ on the event $A$.
Thus, if $\mathbb P(A)>0$, then the L\'evy process $\Xbs$ (which recall was defined in \eqref{levy_XbK}) %
must have the property that $\inf\{t>0;\:\Xbs_t=a\}<\inf\{t>0;\:\Xbs_t>a\}$ with (strictly) positive probability. This means,  as we assumed $a>0$, that the ascending ladder height process of $\Xbs$  must be a compound Poisson process with some atom(s) in its L\'evy measure. %
	 	  Following the arguments of  \cite[p.215]{Kyprianou-14}, this can only  be the case when $\Xbs$  is a compound Poisson process. Since $\Sb_\Sub=(\Sb_{\Sub_t})_{t\geq 0}$ and $\Lev$ are independent, this must mean that both $\Sb_\Sub$ and $\Lev$ are the sum of a compound Poisson process and a possible drift. As we could assume that $\Sub$ has a L\'evy measure with infinite mass, $\Sb$ is a compound Poisson process as otherwise the L\'evy measure of $\Sb_\Sub$ would have infinite mass. This also means that $\Sb_\Sub$  has no drift and thus $\Lev$ has no drift as well because otherwise $\Xbs$  is not a compound Poisson process. But if both $\Lev$ and $\Sb$ are compound Poisson processes, then $\Xb$ cannot creep over the level $a$ because it has to stay in a given state for a (strictly) positive amount of time. Hence $\mathbb P(A)\leq \mathbb P(\Xb_{\ST{\T}{\Xb}}=a)=0$ which completes the proof.
\end{proof}

In the following, we express  $\ST{\T}{\X^{(\Sb_{\Sub_\Lt})}}$ in terms of the subordinator $\Sub$ and the first passage time over $a$ of the L\'evy process $\Xbs$ by using that $\X^{(\Sb_{\Sub_\Lt})}$ is a time-changed L\'evy process.
\begin{lem}\label{lem:T=kT}
For any $a\geq 0$,  $\ST{\T}{\X^{(\Sb_{\Sub_\Lt})}} = \Sub_{\ST{T}{\Xbs}}$ $\P$-a.s. where, using \eqref{levy_XbK},  we have set
\begin{equation*}
\ST{T}{\Xbs} = \inf\{t>0;\: \Xbs_t>a\}.
\end{equation*}
\end{lem}
\begin{proof}
	Since the equality trivially holds if  $\ST{T}{\Xbs}=\infty$ as $\Sub_\infty=\infty$,  we assume from now on  that $\ST{T}{\Xbs}<\infty$. Next,	observe, from its definition, that
	\begin{equation*}
	\ST{\T}{\X^{(\Sb_{\Sub_\Lt})}} = \inf\left\{s>0; \:\Lt_s\in\{t>0; \: \Xbs_t>a\}\right\}.
	\end{equation*}
	By right-continuity of $\Sub$ and the definition of $\Lt$, it is not hard to show that for any $t\geq 0$,
	\begin{equation}\label{rightinverse_L}
	\inf\{s>0; \: \Lt_s>t\} = \Sub_t.
	\end{equation}
	These two observations combined with  the continuity of $\Lt$ yield that
	 \begin{equation*}
	 \ST{\T}{\X^{(\Sb_{\Sub_\Lt})}}  =
	 \begin{cases}
	 \inf\{s>0; \: \Lt_s = \ST{T}{\Xbs} \} & \text{if $\Xbs_{\ST{T}{\Xbs}}>a$}, \\
	 \inf\{s>0; \: \Lt_s > \ST{T}{\Xbs} \} = \Sub_{\ST{T}{\Xbs}} & \text{if $\Xbs_{\ST{T}{\Xbs}}=a$}.
	 \end{cases}
	 \end{equation*}
	  It remains to show that $\mathbb P(A)=0$, where
	  \begin{equation*}
	  A= \left\{ \Xbs_{\ST{T}{\Xbs}}>a \ \text{and} \  \inf\{s>0; \: \Lt_s =  \ST{T}{\Xbs} \}\neq \inf\{s>0;\:\Lt_s > \ST{T}{\Xbs}\} \right\}.
	  \end{equation*}
	  Recalling the notation $\Delta Z_t=Z_t-\lim_{s\uparrow t}Z_s$, note that $A\subset\{ \Delta \Sub_{\ST{T}{\Xbs}}>0\}$, i.e. $\ST{T}{\Xbs}$ is a jump time of $\Sub$ on the event $A$. Since $\Lev$ and $\Sub$ are independent, they jump at different times which means that $\Delta X_{\ST{T}{\Xbs}}=0$ on $A$ and in combination with $\Sb$ having non-decreasing sample paths, we  deduce that $\mathbb P( \Delta \Xbs_{ \ST{T}{\Xbs}}\leq 0,A)=\mathbb P(A)$. But since $\{ \Delta \Xbs_{ \ST{T}{\Xbs}}\leq 0\}\cap\{ \Xbs_{ \ST{T}{\Xbs}}>a \}=\emptyset$, we must have $\mathbb P(A)=0$.
\end{proof}

Next we show that the bivariate process $(\Xbs,\Sub)$ is a L\'evy process, which is crucial for defining our change of measure in order to determine the law of $\Sub_{\ST{T}{\Xbs}}$ and thus $\T_a^{(\Sb)}$. In what follows below  $\langle .,.\rangle$ denotes the standard inner product of $\R^2$.
 \begin{lem}\label{lem_levypair}
 	The two-dimensional stochastic process $\bfXs=(\Xbs,\Sub)=(\Xbs_t,\Sub_t)_{t\geq 0}$ is a two-dimensional L\'evy process whose characteristic exponent is given, for any $\mathbf{z}=(z_1,z_2) \in \R^2$  and $t\geq 0$, by
 	\begin{equation*}
 	\mathbb E \left[ e^{i\langle\mathbf{z}, \bfXs_t\rangle}  \right] = e^{ \mathbf{\Psi}(\mathbf{z})t  }
 	\end{equation*}
 where $\mathbf{\Psi}(\mathbf{z}) =  \Psi(z_1)  -\phi_\Sub(\phi_\Sb(iz_1 ) -iz_2)$.%
 \end{lem}
 \begin{proof}[\textbf{Proof}]
 	Let $n\geq 1$ and write $0=t_0\leq t_1<\ldots<t_{n+1}$, $D_n=\{(s_1,\ldots,s_n)\in[0,\infty)^n; \:s_1\leq  \ldots\leq s_n\}$ and set $B_1,\ldots,B_{n+1}\in\mathcal B(\mathbb R^2)$.
 Then, writing $\mathbf B_n = \cap_{j=1}^{n} \{ \bfXs_{t_{j+1}}-\bfXs_{t_j} \in B_j \}$ and $\nu_{t_1,\ldots, t_{n+1}}(ds_1,\ldots,ds_{n+1}) = \mathbb P(\Sub_{t_1}\in ds_1,\ldots,\Sub_{t_{n+1}}\in ds_{n+1})$, we get
 	\begin{eqnarray}\label{eq_increm_YK}
  \mathbb P  \left(\mathbf B_n \right) \nonumber
 \nonumber	& = & \int_{D_{n+1}}\hspace{-0.6cm}\mathbb P \left( \cap_{j=1}^n \left\{ \left( X_{t_{j+1}}-X_{t_j} - \Sb_{s_{j+1}}+\Sb_{s_j},s_{j+1}-s_j \right) \in B_j \right\} \right) \nu_{t_1,\ldots, t_{n+1}}(ds_1,\ldots,ds_{n+1}) \\
 \nonumber	&= &  \int_{D_{n+1}}\prod_{j=1}^n  \mathbb P \left(  \left( X_{t_{j+1}-t_j}- \Sb_{s_{j+1}-s_j},s_{j+1}-s_j \right) \in B_j \right)   \nu_{t_1,\ldots, t_{n+1}}(ds_1,\ldots,ds_{n+1}) \\
 \nonumber	&= &  \int_{[0,\infty)^{n+1}}  \prod_{j=1}^n  \mathbb P \left(  \left( X_{t_{j+1}-t_j}- \Sb_{u_{j+1}},u_{j+1} \right) \in B_j \right)   \nu_{t_1-t_0,\ldots, t_{n+1}-t_n}(du_1,\ldots,du_{n+1})  \\
 \nonumber	&= & \int_{[0,\infty)^{n+1}}  \prod_{j=1}^n  \mathbb P \left(  \left( X_{t_{j+1}-t_j}- \Sb_{u_{j+1}},u_{j+1} \right) \in B_j \right)  \prod_{i=0}^n \nu_{t_{i+1}-t_i}(du_{i+1})   \\
 \nonumber	&= & \prod_{j=1}^n \int_{[0,\infty)}\mathbb P \left(  \left( X_{t_{j+1}-t_j}- \Sb_{u_{j+1}} , u_{j+1} \right) \in B_j \right) \nu_{t_{j+1}-t_j}(du_{j+1}) \\
 	&= & \prod_{j=1}^n \mathbb P \left( \bfXs_{t_{j+1}-t_j} \in B_j \right),
 	\end{eqnarray}
 	where we used the independence of $X$, $\Sb$ and $\Sub$ in the first and last equality, the stationarity and independence of the increments of $X$ and $\Sb$ together with the independence of $X$ and $\Sb$ in the second equality  and the stationarity and independence of the increments of $\Sub$ in the fourth equality.
 	Equation \eqref{eq_increm_YK} with $n=1$ shows that $\bfXs$ has stationary increments. The stationarity of the increments together with \eqref{eq_increm_YK} shows that $\bfXs$ also has independent increments. Since the process $\bfXs$ has, in addition,  c\`adl\`ag sample paths, we conclude that it is a two-dimensional L\'evy process.
 	Finally, using the independence of $X$, $\Sb$ and $\Sub$, we get, that for any $\mathbf{z}=(z_1,z_2) \in \R^2$  and $t\geq 0$,
 	\begin{eqnarray*}
 	\mathbb E \left[ e^{i\langle\mathbf{z}, \bfXs_t\rangle}  \right]
  	 &= & \mathbb E \left[ e^{iz_1   \Xbs_t +iz_2 {\Sub}_t } \right]=  \mathbb E \left[ e^{iz_1  X_t}\right] \mathbb E \left[ e^{-iz_1  \Sb_{\Sub_t} +iz_2 {\Sub}_t }  \right]  \\
 	&= &  e^{\Psi(z_1 )t}  \int_0^\infty \mathbb E \left[ e^{-iz_1  \Sb_s +iz_2  s} \right] \mathbb P(\Sub_t\in d s)
 	=  e^{\Psi(z_1 )t} \mathbb E \left[ e^{-(\phi_\Sb(iz_1 ) -iz_2 ) \Sub_t }  \right]   \\
 &	= & e^{ ( \Psi(z_1 )  -\phi_\Sub(\phi_\Sb(iz_1 ) -iz_2 ) )t  }
 	\end{eqnarray*}
 which completes the proof of the lemma.
 \end{proof}

  We proceed by providing the change of measure, which requires us to  work in the canonical setting and we refer to \cite{Naj-Nik} for detailed explanations regarding why it can be problematic to do a change
 of measure in a non-canonical filtered probability space.   Let us write  $\mathcal{D}[0,\infty)^2$
for the space of c\`adl\`ag functions $\omega=(\omega_1,\omega_2):[0,\infty)\to\mathbb R^2$. The space $\mathcal{D}[0,\infty)^2$ is equipped with the Skorokhod
topology and its Borel $\sigma$-algebra $\mathcal G$ which is generated by the one-dimensional cylinder sets $\{\omega\in \mathcal{D}[0,\infty)^2;\:\omega(t)\in B\}$ where $t\geq 0$ and
$B\in\mathcal B(\mathbb R^2)$. We let $\mathbf X=(\mathbf X_t^{(1)},\mathbf X_t^{(2)})_{t\geq0}$ denote the canonical process on
$\left( \mathcal{D}[0,\infty)^2, \mathcal G \right)$,
 i.e.~$\mathbf X_t^{(1)}(\omega)=\omega_1(t)$ and $\mathbf X_t^{(2)}(\omega)=\omega_2(t)$ for $\omega=(\omega_1,\omega_2)\in \mathcal{D}[0,\infty)^2$. Next, let us denote by
 $(\mathcal G_t)_{t\geq 0}$  the natural filtration of the canonical process which means that, for any $t\geq 0$, $\mathcal G_t$ is the $\sigma$-algebra over
 $\mathcal{D}[0,\infty)^2$ generated by the one-dimensional cylinder sets $\{\omega\in \mathcal{D}[0,\infty)^2
 ; \:\omega(s)\in B\}$ where $s\in [0,t]$ and $B\in\mathcal B(\mathbb R^2)$.  We set, for any $t\geq 0$, $\mathcal G^+_{t}=\cap_{s>t}\mathcal G_s$. Let now   $\mathbb Q$ be the measure on  $\left( \mathcal{D}[0,\infty)^2, \mathcal G \right)$ defined, for any $G\in\mathcal G$,  by $\mathbb Q(G)= \mathbb P(  \bfXs\in G)$  where we recall that $\bfXs$ was introduced in \eqref{lem_levypair}.
  Note here that as, for any $t\geq 0$, $\bfXs_t:(\Omega,\mathcal F)\to (\mathbb R^2,\mathcal B(\mathbb R^2))$ is measurable and $\bfXs$ takes values in $\mathcal{D}[0,\infty)^2$ we have that $\bfXs:(\Omega,\mathcal F)\to \left( \mathcal{D}[0,\infty)^2, \mathcal G \right)$ is measurable, see e.g.~\cite[Lemma 3.1]{Kallenberg2001}.
 Then,  for any measurable function $\mathbf F:(\mathcal{D}[0,\infty)^2,\mathcal G)\to (\mathbb R,\mathcal B(\mathbb R))$,
 \begin{equation}\label{eq_fromQtoP}
 \begin{split}
 \mathbb E^{\mathbb Q}[ \mathbf F ] = & \int_{\mathcal{D}[0,\infty)^2} \mathbf F(\omega) \mathbb Q(d\omega)
 =   \int_\Omega  \mathbf F \left( \bfXs(\omega)  \right) \mathbb P(d\omega)
 =   \mathbb E \left[ \mathbf F \left( \bfXs \right) \right],
 \end{split}
 \end{equation}
 where  $\mathbb E^\mathbb Q$ denotes the expectation operator associated with $\mathbb Q$, see e.g.~\cite[Lemma 1.22]{Kallenberg2001}.
 From \eqref{eq_fromQtoP} and Lemma \ref{lem_levypair} it is obvious that $\mathbf X$ under $\mathbb Q$ %
 is a two-dimensional L\'evy process with the same characteristic exponent $\mathbf \Psi$ as  $\bfXs$.

An interesting feature of the change of measure, namely \eqref{local_measure_change} below,  that we are going to apply shortly, is that it  changes the law of $\Xbs$ by altering only the law of $\Sub$ but not of $\Lev$ or $\Sb$. Although it looks like a classical Esscher change of measure performed to the one-dimensional L\'evy process $\Sub$, since we need to work with a larger filtration than the canonical analogue of (the right-continuous augmentation of) the natural filtration of $\Sub$ in order for $T_a(\Xbs)$ to become a stopping time, we introduce in \eqref{local_measure_change} a two-dimensional Esscher change of measure disguising as a one-dimensional one. In order to make this entirely clear,
we state below the whole class of Esscher changes of measure corresponding to bivariate L\'evy processes.

 \begin{lem}\label{lem_2d_esscher}
 	Let $\overline{\mathbb Q}$ be a probability measure on $(\mathcal{D}[0,\infty)^2, \mathcal G)$ which is the law of a two-dimensional L\'evy process with characteristic exponent denoted by $\overline{\mathbf{\Psi}}$, i.e.~for any $t\geq 0,$
 	\begin{equation*}
 	  \mathbb E^{\overline{\mathbb Q}} \left[ e^{i \langle\xi,\mathbf{X}_t\rangle} \right] = \int_{\mathcal{D}[0,\infty)^2} e^{i \langle\xi,\mathbf{X}_t\rangle} \overline{\mathbb Q}(\mathrm d\omega)=e^{\overline{\mathbf{\Psi}}(\xi)t}, \quad \xi\in\Xi=\{\xi\in\mathbb C^2; \:\mathbb E^{\overline{\mathbb Q}} \left[  \left| e^{i \langle\xi,\mathbf{X}_1\rangle} \right|  \right] <\infty\}.
 	\end{equation*}
 	Then for any $\bar\xi\in\Xi \cap i\mathbb R^2$, i.e.~$\Re(\bar\xi)=0$,  there exists a unique probability measure $\mathbb Q^{(\bar\xi)}$  on $\left( \mathcal{D}[0,\infty)^2, \mathcal G \right)$ defined  for any $t\geq 0$ and $A\in\mathcal G^+_{t}$ by
 	\begin{equation}\label{meas_change_general}
 	\mathbb Q^{(\bar\xi)}(A) = \mathbb E^{\overline{\mathbb Q}} \left[ \mathbf M^{(\bar\xi)}_t \mathbb{I}_A \right] = \mathbb E^{\overline{\mathbb Q}} \left[  e^{i \langle\bar\xi,\mathbf{X}_t\rangle - \overline{\mathbf{\Psi}}(\bar\xi)t} \mathbb{I}_A \right],
 	\end{equation}
 	where $\mathbf M^{(\bar\xi)}= ( \mathbf M^{(\bar\xi)}_t)_{t\geq 0}$ defined by 	$\mathbf M^{(\bar\xi)}_t = \exp(i \langle\bar\xi,\mathbf{X}_t\rangle - \overline{\mathbf{\Psi}}(\bar\xi)t)$ is a positive, unit-mean martingale with respect to the filtration $(\mathcal G^+_t)_{t\geq 0}$.
 	Moreover, under  $\mathbb Q^{(\bar\xi)}$, $\mathbf{X}$ is a two-dimensional L\'evy process with characteristic exponent given, for any $\mathbf z\in\R^2$, by
 	\begin{equation*}
 	\mathbf{\Psi}_{\bar\xi}(\mathbf z)=\overline{\mathbf{\Psi}}(\mathbf z+\bar\xi)-\overline{\mathbf{\Psi}}(\bar\xi).
 	\end{equation*}
 \end{lem}
 \begin{proof}
 	Let $\bar\xi\in\Xi\cap i\mathbb R^2$. From   \cite[Example 33.14 (with $\eta=i\bar\xi$), Definition 33.3 and Theorem 25.17]{Sato1999} it follows that there exists a unique probability measure $\mathbb Q^{(\bar\xi)}$  on $\left( \mathcal{D}[0,\infty)^2, \mathcal G \right)$ satisfying
 	\eqref{meas_change_general} for all $t\geq 0$ and $A\in\mathcal G_t$.
 	Then for $A\in\mathcal G_{t+}\subset \mathcal G_{t+1/n}$ for any $n\geq 1$, we have by the dominated convergence theorem for conditional expectation and the fact that $\mathbf M^{(\bar\xi)}$ has c\`adl\`ag sample paths,
 	\begin{equation*}
 	\mathbb Q^{(\bar\xi)}(A) = \lim_{n\to \infty} \mathbb E^{\overline{\mathbb Q}} \left[ \mathbf M^{(\bar\xi)}_{t+\frac1n} \mathbb{I}_A \right] =  \mathbb E^{\overline{\mathbb Q}} \left[ \lim_{n\to \infty} \mathbf M^{(\bar\xi)}_{t+\frac1n} \mathbb{I}_A \right] = \mathbb E^{\overline{\mathbb Q}} \left[  \mathbf M^{(\bar\xi)}_{t} \mathbb{I}_A \right]
 	\end{equation*}
 	and so \eqref{meas_change_general} holds for all $t\geq 0$ and $A\in\mathcal G^+_t$. From \eqref{meas_change_general}  and the fact that $\mathbb Q^{(\bar\xi)}$ is a probability measure  on $\left( \mathcal{D}[0,\infty)^2, \mathcal G \right)$, it follows that  $\mathbf M^{(\bar\xi)}$ is a unit-mean martingale with respect to $(\mathcal G^+_t)_{t\geq 0}$.
	Regarding the last statement, from the aforementioned reference it further follows that  $\mathbf{X}$ under $\mathbb Q^{(\bar\xi)}$ is a L\'evy process and via \eqref{meas_change_general} the corresponding characteristic function is easily seen to be, for any $\mathbf z\in\R^2$ and $t\geq 0$,
 	\begin{eqnarray*}
 		\mathbb E^{\mathbb Q^{(\bar\xi)}} \left[ e^{i\langle \mathbf z, \mathbf{X}_t\rangle} \right]  =   \mathbb E^{\overline{\mathbb Q}}  \left[ e^{i\langle \mathbf z+\bar\xi,\mathbf{X}_t\rangle - \overline{\mathbf{\Psi}}(\bar\xi)t }  \right]
 		=  e^{ (\overline{\mathbf{\Psi}}(\mathbf z+\bar\xi)-\overline{\mathbf{\Psi}}(\bar\xi)) t  }.
 	\end{eqnarray*}
 \end{proof}

We have now all the ingredients to complete the proof of Theorem \ref{thm2}.   Recalling  that $\T_a^{(\Sb)}=\ST{\T}{\Xb}$ and using the identity \eqref{rightinverse_L},  we first observe by combining Lemma \ref{lem:T=T} with Lemma \ref{lem:T=kT}, that, for any $q,v\geq 0$ and $a>0$,
\begin{eqnarray}\label{eq:idTT}
  \mathbb E \left[ e^{-q \T_a^{(\Sb)} - v \X^{(\Sb)}_{\T_a^{(\Sb)}} }  \mathbb{I}_{\{ \T_a^{(\Sb)}<\infty \}} \right] &=& \mathbb E \left[ e^{-q \Sub_{\ST{T}{\Xbs}} - v \Xbs_{{\ST{T}{\Xbs}}}}  \mathbb{I}_{\{ \ST{T}{\Xbs}<\infty \}} \right].
 \end{eqnarray}
Next, since $\mathbf X$ under $\mathbb Q$ is a  L\'evy process whose characteristic exponent is $\mathbf{\Psi}(\mathbf{z}) =  \Psi(z_1)  -\phi_\Sub(\phi_\Sb(iz_1 ) +iz_2)$  with $\phi_\Sub$  a Bernstein function, $\mathbf{\Psi}$ admits an analytical continuation to the domain $\{\mathbf{z}=(z_1,z_2) \in \mathbb C^2; \Im(z_1)=0 \textrm{ and } \Im(z_2)>0  \}$. Thus, Lemma \ref{lem_2d_esscher} applied  to the measure $\overline{\mathbb Q}=\mathbb Q$ yields that, for any  $\bar\xi_q=(0,iq)$, $q>0$,  there exists a unique probability measure $\mathbb Q^{(\bar\xi_q)}$ on $(\mathcal{D}[0,\infty)^2,\mathcal G)$ satisfying for any $t\geq 0$ and  $A\in\mathcal G^+_{t}$
 \begin{equation}\label{local_measure_change}
 \mathbb Q^{(\bar\xi_q)}(A)= \mathbb E^{\mathbb Q} \left[ e^{-q \mathbf X^{(2)}_t + \phi_\Sub(q)t} \mathbb{I}_A \right],
 \end{equation}
 where we recall that, under $\mathbb Q$, $\mathbf X^{(2)}$ is a subordinator with Laplace exponent $\phi_\Sub$.
  Since for any $a\geq 0$,  $\ST{T}{\mathbf X^{(1)}}=\inf\{t>0; \: \mathbf X^{(1)}_t>a\}$ is a $(\mathcal G^+_{t})_{t\geq0}$-stopping time, see  \cite[lemma 7.6 and 7.2]{Kallenberg2001}, it follows that, for all $A\in \mathcal G^+_{\ST{T}{\mathbf X^{(1)}}}$,%
 \begin{equation*}
\mathbb E^{\mathbb Q^{(\bar\xi_q)}} \left[ \mathbb I_{\{ A\cap \{\ST{T}{\mathbf X^{(1)}}<\infty\}\}} \right] = \mathbb E^{\mathbb Q} \left[  e^{-q \mathbf X^{(2)}_{\ST{T}{\mathbf X^{(1)}}} + \phi_\Sub(q) \ST{T}{\mathbf X^{(1)}} } \mathbb{I}_{\{A\cap\{ \ST{T}{\mathbf X^{(1)}}<\infty\}\}} \right].
 \end{equation*}
see \cite[Lemma 10.2.2]{rolskietal_book}.
 Hence via \eqref{eq_fromQtoP} and noting that $e^{-\phi_\Sub(q) \ST{T}{\mathbf X^{(1)}} - v \mathbf X^{(2)}_{\ST{T}{\mathbf X^{(1)}}} }$ is $\mathcal G^+_{\ST{T}{\mathbf X^{(1)}}}$-measurable, see \cite[Lemma 7.5]{Kallenberg2001}, we have, using the equality \eqref{eq:idTT} and the characterization of $\mathbf X$ under $\mathbb Q$ given before Lemma \ref{lem_2d_esscher}, that for any $q,v\geq 0$,
 \begin{eqnarray} \label{eq:JLT}
  \mathbb E \left[ e^{-q \T_a^{(\Sb)} - v \X^{(\Sb)}_{\T_a^{(\Sb)}} }  \mathbb{I}_{\{ \T_a^{(\Sb)}<\infty \}} \right]  &= &   \mathbb E^{\mathbb Q} \left[  e^{-q \mathbf X^{(2)}_{\ST{T}{\mathbf X^{(1)}}} -v \mathbf X^{(1)}_{\ST{T}{\mathbf X^{(1)}}}} \mathbb{I}_{\{ \ST{T}{\mathbf X^{(1)}}<\infty\}} \right]  \nonumber \\
 &= & \mathbb E^{\mathbb Q^{(\bar\xi_q)}} \left[  e^{-\phi_\Sub(q) \ST{T}{\mathbf X^{(1)}} -v \mathbf X^{(1)}_{\ST{T}{\mathbf X^{(1)}}}} \mathbb{I}_{\{ \ST{T}{\mathbf X^{(1)}}<\infty\}} \right].%
 \end{eqnarray}
According to Lemma \ref{lem_2d_esscher}, under $\mathbb Q^{(\bar\xi_q)}$, $\mathbf X$ is a two-dimensional L\'evy process with characteristic exponent given, for any $\mathbf z=(z_1,z_2) \in \R^2$, by
 \begin{equation*}%
\mathbf{\Psi}_{\bar\xi_q}(\mathbf z)= \mathbf{\Psi}(\mathbf{z}+\bar\xi_q)-\mathbf{\Psi}(\bar\xi_q) =\Psi(z_1) - { \phi_\Sub(\phi_{\Sb}(iz_1) -iz_2+q) + \phi_\Sub(q) }.
 \end{equation*}
 Hence  $\mathbf X^{(1)}$ is a one-dimensional L\'evy process whose characteristic exponent takes the form, for any $z\in \R$,
$\Psi_{\Ikea}(z)=\mathbf{\Psi}_{\bar\xi_q}((z,0)) =\Psi(z) - { \phi_\Sub(\phi_{\Sb}(iz)+q) + \phi_\Sub(q) }$. %
The identity \eqref{main} then follows from \eqref{eq:JLT} and the Pecherskii-Rogozin identity   stated in \cite[Theorem 49.2]{Sato1999}. By letting $q\downarrow 0$ in  \eqref{main}, we obtain \eqref{main_qis0}. Finally, if $\Xbs$ does not drift to $-\infty$, then acording to \cite[Proposition 37.10]{Sato1999}), $\ST{T}{\Xbs}<\infty$ a.s.~and thus  $\T_a^{(\Sb)}<\infty$ a.s.~by lemmas \ref{lem:T=T} and \ref{lem:T=kT}.

\subsection{Proof of Proposition \ref{cor:sn}}
First, observe that for some (or equivalently any)  $a\geq 0$, that $\P(\T_a^{(\Sb)}<\infty)>0$ if and only if the process $\X^{(\Sb)}$ does not have non-increasing sample paths. These conditions are equivalent to the process $\X^{(\Sb_{\Sub_\Lt})}$, or equivalently, the process  $\Lev - \Sb_{\Sub}$, not having non-increasing sample paths either. However, since the latter process is a L\'evy process, the condition will be satisfied whenever it is not the negative of a subordinator. This is the case if and only if $\sigma^2>0$ or $\int_{1}^{0}|y|\Pi(dy)=\infty$ or its drift, which is easily computed to be $\mathrm{d}_X +\int_{-1}^0 |y| \Pi(dy) - \dt \dk$, is positive, which completes the proof of the first claim.
Regarding the  item \eqref{it:cor_1}, observe that by monotonicity we have, a.s.~$\lim_{v\downarrow a}\T_{v}^{(\Sb)}=\T_{a}^{(\Sb)}$ and  $\lim_{v\uparrow a} \T_{v}^{(\Sb)}=\inf\{t>0; \:\X^{(\Sb)}_t\geq a\}$ for all $a\geq 0$ and so the process $(\T_a^{(\Sb)})_{a\geq 0}$ has right-continuous sample paths with left limits. Then, in order  for $(\T_a^{(\Sb)})_{a\geq 0}$ to be a subordinator killed at rate $\phi_{{\Sub_{\Sb}^{\triangleright 0}}}(0)$, we need to show, see \cite[Section III.1]{Bertoin-96}, that, for any $a\geq 0$, $\mathbb P(\T_a^{(\Sb)}<\infty) = \exp(-\phi_{{\Sub_{\Sb}^{\triangleright 0}}}(0)a)$  and for any $h\geq  0$, under $\mathbb P(\cdot|\T_a^{(\Sb)}<\infty)$, $\T_{a+h}^{(\Sb)}-\T_a^{(\Sb)}$ is independent of $(\T_v^{(\Sb)})_{0\leq v\leq a}$ and has the same law as $\T_h^{(\Sb)}$ under $\mathbb P$.
 Note that by Lemma \ref{lem:T=T}, $\mathbb P \left( \T_a^{(\Sb)}=\ST{\T}{\X^{(\Sb_{\Sub_\Lt})}} \right)=1$ for all $a\geq0$, which implies, as both $\T_a^{(\Sb)}$ as well as $\ST{\T}{\X^{(\Sb_{\Sub_\Lt})}}$ are right-continuous in $a$ that, $\mathbb P \left( \T_a^{(\Sb)}=\ST{\T}{\X^{(\Sb_{\Sub_\Lt})}}, \ \forall a\geq 0 \right)=1$. Hence we are done if we show that $\left( \ST{\T}{\X^{(\Sb_{\Sub_\Lt})}} \right)_{a\geq 0}$ is a subordinator killed at rate $\phi_{{\Sub_{\Sb}^{\triangleright 0}}}(0)$.
To this end, for any $a\geq 0$, by Lemma \ref{lem:T=kT} and a well-known result for spectrally negative L\'evy processes, see e.g.~\cite[Corollary 3.13]{Kyprianou-14}, one has that
\begin{equation*}
 \mathbb P(\ST{\T}{\X^{(\Sb_{\Sub_\Lt})}} <\infty) = \mathbb P(\Sub_{\ST{T}{\Xbs}} <\infty) = \mathbb P({\ST{T}{\Xbs}} <\infty)
 = e^{-\phi_{{\Sub_{\Sb}^{\triangleright 0}}}(0)a}.
\end{equation*}
Then, since $(\Xbs,\Sub)$ is a bivariate L\'evy process by Lemma \ref{lem_levypair} we are in the setting of Proposition \ref{prop:reg} where with the notation therein $(\Lev,\Sub)=(\Xbs,\Sub)$.  Since the filtration $(\widetilde{\mathcal{F}}_t)_{t\geq 0}$, as defined again in Proposition \ref{prop:regenerative}, is right-continuous and $\X^{(\Sb_{\Sub_\Lt})}$ is $(\widetilde{\mathcal{F}}_t)_{t\geq 0}$-adapted (recall that $\X^{(\Sb_{\Sub_\Lt})}_t=\Xbs_{\Lt_t}$), we have that, for any $a\geq 0$, $\ST{\T}{\X^{(\Sb_{\Sub_\Lt})}}$ is an $(\widetilde{\mathcal{F}}_t)_{t\geq 0}$-stopping time, see  \cite[Lemmas 7.6 and 7.2]{Kallenberg2001}. Moreover,  by Lemma \ref{lem:T=T}, one has that $[\ST{\T}{\X^{(\Sb_{\Sub_\Lt})}}]\subseteq \Rs\cup\{\infty\}$ a.s.~for any $a\geq 0$, and,  by lack of upward jumps $\X^{(\Sb_{\Sub_\Lt})}_{\ST{\T}{\X^{(\Sb_{\Sub_\Lt})}}}=a$ if $\ST{\T}{\X^{(\Sb_{\Sub_\Lt})}}<\infty$. Therefore, one gets, given $\ST{\T}{\X^{(\Sb_{\Sub_\Lt})}}<\infty$, that
\begin{eqnarray*}
{\T}_{a+h}(\X^{(\Sb_{\Sub_\Lt})}) - \ST{\T}{\X^{(\Sb_{\Sub_\Lt})}}
&= & \inf\{t>\ST{\T}{\X^{(\Sb_{\Sub_\Lt})}}; \: \X^{(\Sb_{\Sub_\Lt})}_t  > a+ h\} - \ST{\T}{\X^{(\Sb_{\Sub_\Lt})}} \\
&= & \inf \left\{ t>\ST{\T}{\X^{(\Sb_{\Sub_\Lt})}};\: \X^{(\Sb_{\Sub_\Lt})}_{t} - \X^{(\Sb_{\Sub_\Lt})}_{\ST{\T}{\X^{(\Sb_{\Sub_\Lt})}}} > h \right\} - \ST{\T}{\X^{(\Sb_{\Sub_\Lt})}} \\
&= & \inf \left\{ t>0;\: \X^{(\Sb_{\Sub_\Lt})}_{t+\ST{\T}{\X^{(\Sb_{\Sub_\Lt})}}} - \X^{(\Sb_{\Sub_\Lt})}_{\ST{\T}{\X^{(\Sb_{\Sub_\Lt})}}} > h \right\}.
\end{eqnarray*}
This combined with Proposition \ref{prop:reg} yield that, for any $a,h\geq 0$, under $\mathbb P(\cdot|\ST{\T}{\X^{(\Sb_{\Sub_\Lt})}}<\infty)$, ${\T}_{a+h}(\X^{(\Sb_{\Sub_\Lt})}) - \ST{\T}{\X^{(\Sb_{\Sub_\Lt})}}$ is independent of $\widetilde{\mathcal{F}}_{\ST{\T}{\X^{(\Sb_{\Sub_\Lt})}}}$ and has the same law as $\T_h(\X^{(\Sb_{\Sub_\Lt})})$ under $\mathbb P$. As $\T_v(\X^{(\Sb_{\Sub_\Lt})})$ is $\widetilde{\mathcal{F}}_{\ST{\T}{\X^{(\Sb_{\Sub_\Lt})}}}$-measurable for all $0\leq v\leq a$ since $\T_v(\X^{(\Sb_{\Sub_\Lt})})$ is increasing in $v$ and an $(\widetilde{\mathcal{F}}_t)_{t\geq 0}$-stopping time, we conclude that  $(\ST{\T}{\X^{(\Sb_{\Sub_\Lt})}})_{a\geq 0}$ is a subordinator killed at rate $\phi_{{\Sub_{\Sb}^{\triangleright 0}}}(0)$.
Since, for any $q\geq 0$, $z\mapsto \Psi_{\Ikea}(z)$ is the characteristic exponent of a spectrally negative L\'evy process we have that its positive Wiener-Hopf factor $\Phi_\Ikea(p;z)$ is given by $\Phi_\Ikea(p;z)=\frac{\phi_{\Ikea}(p)}{\phi_{\Ikea}(p)-iz}$ for $p>0$ and $\Im(z)\geq 0$, see e.g.~\cite[Equation (8.4)]{Kyprianou-14}. Hence, performing a simple Laplace inversion to the composite Wiener-Hopf identity \eqref{main}, one has, for any $q>0$ and $a\geq 0$,
 \[\E\left[ e^{-q \T^{(\Sb)}_a } \mathbb I_{\{\T^{(\Sb)}_a <\infty\}} \right]= e^{-\phi_{\Ikea}(\phi_\Sub(q))a} = e^{-{\phi}_{\T^{(\Sb)}}(q)a}. \] The case  $q=0$  follows by taking the limit on both sides. Since ${\phi}_{\T^{(\Sb)}}$ is the Laplace exponent of a subordinator, it must be a Bernstein function, which completes the proof of item \eqref{it:cor_1}.
 The first part of item \eqref{it:csn2} follows easily from  the previous one whereas the second part is obtained after observing  that the mapping $\Psi_{\Sub}(-iu)= u^{\frac{\beta}{\alpha}}, u\geq 0,$ with $0<\alpha<1<\beta \leq 2\alpha$, is indeed on the one hand the Laplace exponent of a spectrally negative L\'evy process, namely the one of a $\frac{\beta}{\alpha}$-stable L\'evy process with $1<\frac{1}{\alpha}<\frac{\beta}{\alpha}<2$, and,  on the other hand the inverse of $\phi \circ \phi_{\Sub}(u)=u^{\frac{\alpha}{\beta}}$, where $\phi$ is the inverse of $\Psi(-iu)=u^{\beta}$, $u\geq 0$. This provides the proof of this item. %
Finally, from lemmas \ref{lem:T=T} and \ref{lem:T=kT}, we get that, under $\mathbb P$, for any $a\geq 0$, $\T_a=\Sub_{T_a(\Lev)}$ and $\widehat\T_{-a}=\Sub_{\widehat T_{-a}(\Lev)}$ where $\widehat T_{a}(\Lev) = \inf\{t>0;\:  {\Lev}_t < a\}$. Hence by spatial homogeneity and the fact that $\Sub$ is increasing, we have, for $0\leq x\leq a$ and $q\geq 0$,
\begin{eqnarray*}
\E_x \left[e^{-q \T_{a} }\mathbb{I}_{\{\T_{a} <\widehat{\T}_0\}} \right]
&= & \E \left[e^{-q \T_{a-x} }\mathbb{I}_{\{\T_{a-x} <\widehat{\T}_{-x}\}} \right] \\
&= & \E \left[e^{-q \Sub_{T_{a-x}(\Lev)} } \mathbb{I}_{\{ \Sub_{T_{a-x}(\Lev)} <\Sub_{\widehat T_{-x}(\Lev)} \}} \right] \\
&= & \E_x \left[ e^{-q \Sub_{T_{a}(\Lev)} } \mathbb{I}_{\{ T_{a}(\Lev) < \widehat T_{0}(\Lev) \}} \right] \\
&= & \E_x \left[ e^{-\phi_\Sub(q) T_{a}(\Lev) } \mathbb{I}_{\{ T_{a}(\Lev) < \widehat T_{0}(\Lev) \}} \right] \\
&= &  \frac{W^{(\phi_{\Sub}(q))}(x)}{W^{(\phi_{\Sub}(q))}(a)}
\end{eqnarray*}
where for the penultimate equality we used the independence of $X$ and $\Sub$ and for the last one we used \cite[Theorem 8.1(iii)]{Kyprianou-14}. Similarly one proves the last identity.

\subsection{Proof of Proposition \ref{cor:ident}}
 Since it is assumed that $\Sub$ is an $\alpha$-stable subordinator, we have, in particular,
 that  for any $s>0$, $\Sub_s \stackrel{(d)}{=} s^{\frac{1}{\alpha}}\Sub_1$.
 Then, using  Lemma \ref{lem:T=kT} with $\Sb=0$, by conditioning and using the independence of the involved random variables, we obtain that, for any $t\geq0$ and $x>0$,
\begin{eqnarray}
\P_{-x}(\T_0 \in dt)&=&\int_{0}^{\infty}\P(\Sub_s \in dt)\P_{-x}(\Tm_0 \in ds) \nonumber \\
&=&\int_{0}^{\infty}\P(s^{\frac{1}{\alpha}}\Sub_1 \in dt)\P_{-x}(\Tm_0 \in ds) \label{eq:abs}\\
&=& \P_{-x}(\Sub_1 \Tm_0^{\frac{1}{\alpha}} \in dt)\nonumber
\end{eqnarray}
which completes the proof of the  first identity in law, that is
 \beq\label{eq:idT}
\T_0 \stackrel{(d)}{=} \Sub_1 \times \Tm_0^{\frac{1}{\alpha}}.
\eeq
The second claim follows readily from  \eqref{eq:abs} and the fact that the law of $\Sub_s$ is absolutely continuous for all $s> 0$, see \cite[Chap.~3.14]{Sato1999}. Next, the fact that $\P_{-x}(\Tm_0>0)>0, x>0,$ combined with the identity \eqref{eq:idT}  and  $\E[\Sub_1^{\alpha}]=\infty$, see \eqref{eq:ms} below, imply the first claim of item  \eqref{it:corm} that is $\E_{-x}[\T_0^{\alpha}]=\infty$. Next, we recall that Doney and Maller \cite[Theorem 2]{Doney-Maller} showed that the condition $\int_0^{1}  e^{\int_1^{\infty}e^{-qt}\P(Z_t\leq 0) \frac{dt}{t}}\frac{dq}{q^{\delta}} < \infty$ for some $0<\delta<1$ is equivalent to  $\E_{-x}\left[\Tm_0^{\delta}\right]<\infty$. Then, classical results on Mellin transform  yield that, for any $x>0$, the mapping
\begin{equation}\label{eq:ht} z\mapsto \E_{-x}\left[\Tm_0^{\frac{z}{\alpha}}\right] \textrm{ is holomorphic in the strip } \mathbb{S}_{(0,\alpha \delta)}= \left\{z\in \C;\: 0<\Re(z)<\alpha \delta\right \}
\end{equation}
with
\begin{equation} \label{eq:bm} \left|\E_{-x}\left[\Tm_0^{\frac{z}{\alpha}}\right]\right|\leq \E_{-x}\left[\Tm_0^{\frac{\Re(z)}{\alpha}}\right]<\infty \textrm{ for any  } z \in \mathbb{S}_{(0,\alpha \delta)}.
\end{equation}
Hence,  recalling, from e.g.~\cite[(25.5)]{Sato1999}, that, with $\phi(u)=u^{\alpha}$,
\begin{equation} \label{eq:ms}
z\mapsto  \E[\Sub_1^{z}] = \frac{\Gamma(1-\frac{z}{\alpha})}{\Gamma(1-z)} \textrm{ is holomorphic in the left half-plane } \mathbb{S}_{(-\infty,\alpha )},
\end{equation}
we deduce, by means of the identity \eqref{eq:idT}, the property \eqref{eq:ht} and since $\delta \in (0,1)$, that
\begin{equation} \label{eq:hto} z\mapsto \E_{-x}\left[\T_0^{z}\right]=\E[\Sub_1^{z}] \E_{-x}\left[\Tm_0^{\frac{z}{\alpha}}\right] \textrm{ is holomorphic in the strip } \mathbb{S}_{(0,\alpha \delta)}. \end{equation}
Moreover, recalling the  Stirling formula of the gamma function, that for fixed $a\in\R$,
\begin{equation} \label{eq:asympt_gamma}
|\Gamma(a+ib)|\sim C |b|^{a-\frac12}e^{-\frac{\pi}{2}|b|} \textrm{ as } |b| \to \infty,
\end{equation}
with $C=C(a)>0$, simple algebra entails that  for any  $z =a+ib \in \mathbb{S}_{(0,\alpha \delta)}$ there exists $C(a,\alpha)>0$ such that
\begin{equation}\label{eq:Tbound}
\left|\E_{-x}\left[\T_0^{z}\right]\right|\leq C(a,\alpha) \E_{-x}\left[\Tm_0^{\frac{a}{\alpha}}\right] |b|^{\frac{\alpha-1}{\alpha} a}e^{-\frac{1-\alpha}{\alpha}\frac{\pi}{2}|b|} \textrm{ as } |b| \to \infty,
\end{equation}
where we used the upper bound \eqref{eq:bm}. Hence, \eqref{eq:hto}  combined with \eqref{eq:Tbound} justifies that one can use Mellin inversion formula, see  \cite[Section 1.7.4]{Patie-Savov-16},  to get, for any $0<a <\alpha \delta$, the following Mellin-Barnes representation of the density
 \begin{equation*}%
  	f_{\T_0}(t)=\frac{1}{2\pi i}\int^{a+i\infty}_{a-i\infty}t^{-z}\frac{\Gamma(1-\frac{z}{\alpha})}{\Gamma(1-z)}\E_{-x}\left[\Tm_0^{\frac{z}{\alpha}}\right]dz,
  	 	\end{equation*}
  where the integral is absolutely convergent for any $t>0$. Using this representation and the bound \eqref{eq:Tbound} combined with a
  dominated convergence argument yields that $f_{\T_0}$ admits an analytical extension to the sector $\C_{(\frac{1-\alpha}{\alpha}\frac{\pi}{2})}=\{z\in\C;\,|\arg z|<\frac{1-\alpha}{\alpha}\frac{\pi}{2}\}$.  Since from \eqref{eq:Tbound} again, we get that, for any $n \in \mathbb{N}$, the mapping $z \mapsto (z+n)^n \E_{-x}\left[\T_0^{z}\right]$ is absolutely integrable and uniformly decaying along the complex lines of the strip $\mathbb{S}_{(0,\alpha \delta)}$ and hence for any $0<a <\alpha \delta$  and $t>0$,
 \begin{equation}\label{eq:MITd2}
  	f_{\T_0}^{(n)}(t)=\frac{(-1)^n}{2\pi i}\int^{a+i\infty}_{a-i\infty}t^{-z-n}\frac{\Gamma(z+n)}{\Gamma(z)}\frac{\Gamma(1-\frac{z}{\alpha})}{\Gamma(1-z)}\E_{-x}\left[\Tm_0^{\frac{z}{\alpha}}\right]dz.
  	 	\end{equation}
  Moreover, a dominated convergence argument gives that  for all $n \in \mathbb{N}$, $f^{(n)}_{\T_0} \in C_0(\R^+)$.   Next, assuming that %
    $\E_{-x}\left[\Tm_0^{1+\delta}\right]<\infty$ for some $\delta>0$. As above this yields that now $z\mapsto \E_{-x}\left[\Tm_0^{\frac{z}{\alpha}}\right]$  is holomorphic in the strip $ \mathbb{S}_{(0,\alpha (1+\delta))}= \left\{z\in \C;\: 0<\Re(z)<\alpha (1+\delta)\right \}$ and thus the representation \eqref{eq:MITd2} also holds for $f_{\T_0}^{(n)}$ for any $n \in \mathbb{N}$ and any $0<a <\alpha =\min(\alpha,\alpha(\delta+1))$.  By shifting the contour to the imaginary line $\Re(z)=a_{\alpha}$ where $\alpha < a_{\alpha} < \alpha (\delta+1)$,  one gets by an application of the Cauchy's residue Theorem, that
 \begin{equation*}%
  	f_{\T_0}^{(n)}(t)=  \frac{(-1)^{n}\Gamma(\alpha+n)}{\Gamma(\alpha)\Gamma(1-\alpha)}\E_{-x}\left[\Tm_0\right]t^{-\alpha -n }+\frac{(-1)^n}{2\pi i}\int^{a_{\alpha}+i\infty}_{a_{\alpha}-i\infty}t^{-z-n}\frac{\Gamma(z+n)}{\Gamma(z)}\frac{\Gamma(1-\frac{z}{\alpha})}{\Gamma(1-z)}\E_{-x}\left[\Tm_0^{\frac{z}{\alpha}}\right]dz.
  	 	\end{equation*}
  Next, since plainly $z \mapsto H_n(z)= \frac{\Gamma(z+n)}{\Gamma(z)}\frac{\Gamma(1-\frac{z}{\alpha})}{\Gamma(1-z)}\E_{-x}\left[\Tm_0^{\frac{z}{\alpha}}\right]$ is absolutely integrable on $\Re(z)=a_{\alpha}$, the Riemann-Lebesgue lemma and the fact that $a_{\alpha}>\alpha$ yield that
 \begin{eqnarray*}
 \lim_{t\to\infty}t^{\alpha + n }\int^{a_{\alpha}+i\infty}_{a_{\alpha}-i\infty}t^{-z-n} H_n(z)dz &=&  \lim_{t\to\infty}t^{\alpha-a_{\alpha} } \int^{\infty}_{-\infty}e^{ib \ln t} H_n(a_{\alpha}+ib)db =0.
  	 	\end{eqnarray*}
  Then, by means of the Euler reflection formula for the gamma function, we obtain  the asymptotic \eqref{eq:MITda} and we complete the proof by integrating the case $n=0$ in \eqref{eq:MITda}.

\bibliographystyle{plain}

\end{document}